\documentclass[11pt]{article}
\usepackage{bbm,times,bm}
\usepackage{soul}
\usepackage[numbers]{natbib}

\usepackage{mathrsfs}
\usepackage{amsfonts}
\usepackage{amssymb}
\usepackage{verbatim,amsfonts,amssymb, mathrsfs, amsmath,   amssymb, float,amsthm,enumitem,amsthm}

\usepackage[margin=0.98in]{geometry}

\newtheorem{theorem}{Theorem}[section]
\newtheorem{lemma}[theorem]{Lemma}
\newtheorem{prop}[theorem]{Proposition}
\newtheorem{cor}[theorem]{Corollary}

\newtheorem*{question}{Question}

\theoremstyle{definition}
\newtheorem{definition}[theorem]{Definition}
\newtheorem{remark}[theorem]{Remark}
\newtheorem{example}[theorem]{Example}
\newtheorem*{claim}{Claim}

\floatstyle{ruled}

\def\and{\cap}

\def\deg{\hbox{\rm{deg}}}

 \def\dim{\hbox{\rm{dim}}}

 \DeclareMathOperator{\diffspec}{diffspec}
 \DeclareMathOperator{\Char}{char}
 
 \DeclareMathOperator{\ddim}{\text{$\delta$-dim}}
 \DeclareMathOperator{\ord}{\text{ord}}
  
  \DeclareMathOperator{\deltrdeg}{\text{$\delta$-tr.deg}}
    \DeclareMathOperator{\trdeg}{\text{tr.deg}}
  \DeclareMathOperator{\kol}{Kol}
  \DeclareMathOperator{\Id}{id}
   \DeclareMathOperator{\Ker}{ker}
 
\usepackage[dvipsnames]{xcolor}
\usepackage[colorlinks=true,linkcolor=RoyalBlue,citecolor=PineGreen,urlcolor=blue]{hyperref}

\author{Wei Li, Alexey Ovchinnikov, Gleb Pogudin, and Thomas Scanlon\smallskip\\ \small
KLMM,
   Academy of Mathematics and Systems Science,
   Chinese Academy of Sciences\\\small
    No.55 Zhongguancun East Road, 100190, Beijing, China\\\small
\href{mailto:liwei@mmrc.iss.ac.cn}{liwei@mmrc.iss.ac.cn} \vspace{0.05in}\\ \small
CUNY Queens College, Department of Mathematics,
65-30 Kissena Blvd, Queens, NY 11367, USA \\ \small
CUNY Graduate Center, Mathematics and Computer Science, 365 Fifth Avenue,
New York, NY 10016, USA\\ \small
\href{mailto:aovchinnikov@qc.cuny.edu}{aovchinnikov@qc.cuny.edu} \vspace{0.05in} \\  \small Courant Institute of Mathematical Sciences, New York University, New York, NY 10012, USA\\ \small
\href{mailto:pogudin@cims.nyu.edu}{pogudin@cims.nyu.edu}\\ \small University of California at Berkeley, Department of Mathematics, Berkeley, CA 94720,\\ 
\small\href{mailto:scanlon@math.berkeley.edu}{scanlon@math.berkeley.edu}}

\begin{document}
\date{}
 \title{Elimination of unknowns for systems of algebraic differential-difference equations}

 \maketitle
 \begin{abstract}
We establish effective elimination theorems for differential-difference equations.  Specifically, 
we find a computable function $B(r,s)$ of the natural number parameters $r$ and $s$ so that 
for any system of algebraic differential-difference equations in the variables
$\bm{x} = x_1, \ldots, x_q$ and $\bm{y} = y_1, \ldots, y_r$ each of which has order and
degree in $\bm{y}$ bounded by $s$ over a differential-difference field, there is a non-trivial 
consequence of this system involving just the $\bm{x}$ variables if and only if such a 
consequence may be constructed algebraically by applying no more than $B(r,s)$ iterations of the 
basic difference and derivation operators to the equations in the system.  We relate this finiteness
theorem to the problem of finding solutions to such systems of differential-difference equations in 
rings of functions showing that a system of differential-difference equations
over $\mathbb{C}$ is algebraically consistent if and only if it has solutions in a certain 
ring of germs of meromorphic functions.

 \end{abstract} 
 \section{Introduction}
Differential-difference equations, or what are sometimes called delay 
 differential
 equations, especially
 when the independent variable represents time, are ubiquitous in applications.  See for instance~\cite{apps-bio-DDE}
 and the collection it introduces for a discussion of applications of delay 
differential
 equations in biology,
 the discussion of the  follow-the-leader model in~\cite{crowds} for the use of differential-difference 
 equations to model crowd behavior, and~\cite{DDEapps} for a thorough discussion of theory of 
 delay differential equations and their applications to population dynamics and other fields.  
 Much work has been undertaken in the analysis of the behavior of the solutions of these equations.  We take
 up and solve parallel problems.  First, we address the problem of determining 
 the consistency of a systems of algebraic differential-difference equations, and more generally, of
 eliminating variables  
for
 such a system of equations.   Secondly, we 
 ask and answer the question of what structures should serve as the universal differential-difference
 rings in which we seek our solutions to these equations.   
 
 Our solution to the first problem, that is, of performing effective elimination for systems of 
 differential-difference equations, is achieved by reducing the problem for 
 differential-difference equations to one for ordinary polynomial equations to which 
standard methods in computational algebra may be applied.  Let us state our main theorem, Theorem~\ref{thm:main}, now. Precise definitions are given in Section~\ref{sec:basic-notation}.     We show that there is 
a computable function $B(r,s)$ of the natural number parameters $r$ and $s$ so that whenever one is 
given tuples of variables $\bm{x} = x_1,\ldots,x_q$  and $\bm{y} = y_1, \ldots, y_r$ and 
a set $F$ of differential-difference polynomials in these variables each of which has  order and degree in $\bm{y}$ 
bounded by $s$ over some 
differential-difference field, then the differential-difference ideal generated by $F$, that is,
the ideal generated by the elements of $F$ and all of their transforms under iterated applications of
the distinguished difference and derivation operators, contains a nontrivial differential-difference 
polynomial in just the $\mathbf{x}$ variables if and only if the ordinary ideal generated by 
the transforms of elements of $F$ of order at most $B = B(r,s)$ already contains such a nontrivial
differential-difference polynomial in $\bm{x}$.   In particular, taking $q = 0$, this 
gives a procedure to test the consistency of a system of differential-difference equations.

The reader may rightly object that rather than giving a method for determining consistency of 
such a system of equations, what we have really done is to give a method for testing whether there is 
an explicit algebraic obstruction to the existence of a solution.  In what sense must a solution 
actually exist if there is no such algebraic obstruction?  This brings us to our second question of 
where to find the solutions.   We address this problem in Section~\ref{sec:seqfun}, in which 
we begin by proving an abstract Nullstellensatz theorem to the effect that solutions may 
always be found in differential-difference rings of sequences constructed from 
differential-difference fields.  Of course, in practice, one might expect that the differential-difference
equations describe \emph{functions} 
for which
the difference or delay operator takes the form $\sigma(f)(t) = 
f(t - \tau)$ (for some fixed parameter $\tau$) and the derivation operator is given by usual differentiation
so that $\delta f = \frac{df}{dt}$.   We establish with 
Proposition~\ref{prop:strongnullstellensazt_germs} that certain rings of 
\emph{germs} of meromorphic functions serve as universal differential-difference rings in the 
sense that every algebraically consistent system of differential-difference equations over 
$\mathbb{C}$ has solutions in these rings of germs.  As a complement to this positive result, 
we 
show
that there are algebraically consistent differential-difference equations 
that
cannot be 
solved in any ring of meromorphic \emph{functions} (as opposed to germs).  

The method of proof of our main theorem is modeled on the approach taken by three of the
present authors in~\cite{OPS} for algebraic difference equations in that we modify and 
extend the decomposition-elimination-prolongation (DEP)  method.  However, we encounter 
some very substantial obstacles in extending these arguments to the differential-difference
context.   First of all, we argue by reducing from differential-difference equations to 
\emph{differential} equations and then complete the reduction to algebraic equations using 
methods in computational differential algebra.   Differential algebra in the sense of Ritt and 
Kolchin, especially the Noetherianity of the Kolchin topology, substitute for classical commutative
algebra and properties of the Zariski topology, but there are essential distinctions preventing a 
smooth substitution.  Most notably, in the computation of a bound for the length of a 
possible skew-periodic train in~\cite{OPS}, one argues by induction on the \emph{co}dimension of 
a certain subvariety.  In that algebraic case, since the ambient dimension is finite, such an 
inductive argument is well-founded.  This would fail 
in the case at hand with differential algebraic varieties.  To deal with this difficulty, 
we must argue with a much subtler induction stepping through a decreasing (and hence finite) chain 
of Kolchin polynomials.   

With other steps of the argument, we must invoke or prove delicate 
theorems on computational differential algebra for which the corresponding results for 
ordinary polynomial rings are fairly routine.  For instance, a key step in the calculation of
our bounds involves computing upper bounds on the number of irreducible components of 
a differential algebraic variety given bounds on  
such
associated 
parameters  
as the degrees of defining equations and the dimensions of certain ambient varieties.  In the 
algebraic case, these bounds are provided by B\'{e}zout-type theorems.  Here, we work to establish
such bounds with Proposition~\ref{prop:main}.

For the most part, the methods we employ could be used to prove analogous theorems for partial 
differential-difference equations.  That is, we would work with a ring $\mathcal{R}$ equipped with a
ring endomorphism $\sigma:\mathcal{R} \to \mathcal{R}$ and finitely many commuting derivations 
$\delta_1, \ldots, \delta_n:\mathcal{R} \to \mathcal{R}$ each of which commutes with $\sigma$.   
Our arguments go through verbatim in this case up to the point of the computation of bounds on the 
number of irreducible components of a differential variety and to our knowledge it is an open problem
whether such bounds exist, much less what the bounds might be.  For the 
bounds we compute, it helps that for ordinary differential fields, the coefficients
of the Kolchin polynomial are geometrically meaningful.  This is not so for 
partial differential fields, but the bounds on these coefficients obtained in~\cite{Leon-Sanchez} 
should make it possible to extract explicit bounds analogous to $B(r,s)$ in the 
case of partial differential-difference equations 
once the issue of bounding the number of irreducible 
components has been resolved.

Other potential generalizations 
present themselves, but they, too, lie outside the reach of our present methods.   For example, 
one might wish for an elimination theorem for differential-difference equations in positive
characteristic, especially as there is no restriction on the characteristic for the elimination theorem 
for algebraic difference equations, but many difficulties arise in positive characteristic, starting 
with the non-Noetherianity of the corresponding differential algebraic topology.   In another direction, 
one might wish to allow for several commuting difference operators, corresponding, for example, to allowing
delays of various scales in the delay differential equations.   Based on our preliminary investigations, we 
expect the ultimate theorems to have a fundamentally different character for two or more difference 
operators.  In particular, even for algebraic difference equations, we know of examples of consistent
systems of difference equations with two commuting difference operators for which there are no skew-periodic
solutions.   On its own, this does not rule out the possibility of an effective elimination theorem, but 
it does mean that the approach we take here cannot be applied.

Finally, as a matter of proof technique, our approach is to reduce from equations in several operators to 
equations in fewer operators and so on until we reach purely algebraic equations.   We expect that it may 
be possible to compute better bounds by making use of integrability conditions in the DEP method to 
reduce directly from equations with operators to algebraic equations.  We do not pursue this idea here.

Some work on elimination theory for differential-difference equations appears in the literature, though the 
known results do not cover the problems we consider.  In~\cite{MS03} algorithms for 
computing analogues of Gr\"{o}bner bases in certain differential-difference algebras are 
developed.  These algebras are rings of linear differential-difference operators.  So, the
resulting elimination theorems are appropriate for linear equations, but not for the 
nonlinear differential-difference equations we consider. 
Gr\"{o}bner bases of a different kind for rings of differential-difference polynomial
rings are considered in~\cite{Levin2000, ZW2008} with the aim of computing generalized
Kolchin polynomials.  In these papers, the invariants are computed in 
\emph{fields}, but as one sees in applications and as we will show in Section~\ref{sec:seqfun}, 
one must consider possible solutions in 
rings of sequences or of functions in which there are many zero divisors. 

In the 
papers~\cite{C1970} and~\cite{GVYZ2009}, characteristic set methods for 
differential-difference rings are developed.  While one might imagine that these techniques
may be relevant to the problems we consider, it is not clear how to apply them directly
as once again, generally, characteristic set methods are best adapted to studying solutions of
such equations in fields.
In~\cite{BM07} the model theory of 
differential-difference fields of characteristic zero is worked out.   The results include a 
strong quantifier simplification theorem from which one could deduce an effective elimination 
theorem in our sense.  In~\cite{MS2014} such quantifier simplification theorems were proven for
fields equipped with several operators.  An overview of the model theory of
fields with operators may be found in~\cite{C15}.  All of these quantifier elimination theorems
for difference and differential fields very strongly
use the hypothesis that the solutions are sought in a \emph{field}.  Already at the 
level of algebraic difference equations, the results of~\cite{HP07} show that if we allow
for solving our equations in rings of sequences or their like, then the corresponding 
logical theory will be undecidable.  In particular, no quantifier elimination theorem of 
the kind known for differential-difference fields can hold.   This makes our 
effective elimination theorem all that more remarkable.

This paper is organized as follows.  We start in Section~\ref{sec:basic-notation} by introducing
the technical definitions we require to state our main theorems.  These
theorems are then announced in Section~\ref{sec:main-result}.   With Section~\ref{sec:defproof}
we give the definitions of the technical concepts used in our proofs.  We
deal with the question of where we should seek the solutions to our differential-difference
equations in Section~\ref{sec:seqfun}.  The proof of our main theorem occupies
Section~\ref{sec:proofmain}.

%%%%%%%%%%%%%%%%%%%%%%%%%%%
 
 \section{Basic notation to state the main result}
 \label{sec:basic-notation}
\begin{definition}[Differential-difference rings]
\begin{itemize}
\item[]
    \item 
 A differential-difference 
 ring $(\mathcal R, \delta, \sigma)$ is a commutative ring $\mathcal R$ endowed with a derivation  $\delta$ and 
 an endomorphism $\sigma$ such that $\delta\sigma=\sigma\delta$. 
 \item 
 For simplicity of the notation,  we say $\mathcal R$ is  a $\delta$-$\sigma$-ring.
 \item
 When $\mathcal R$ is additionally a field, it is called a $\delta$-$\sigma$-field. 
 \item
 If $\sigma$ is an automorphism of $\mathcal R$, $\mathcal R$ is  called an inversive $\delta$-$\sigma$-ring, or simply a $\delta$-$\sigma^\ast$-ring.
 \item 
If $\sigma = \operatorname{id}$, $\mathcal{R}$ is called a $\delta$-ring or differential ring.
 \item
 Given two $\delta$-$\sigma$-rings $\mathcal R_1$ and $\mathcal R_2$,
 a   homomorphism $\phi: \mathcal R_1\longrightarrow \mathcal R_2$ is called a $\delta$-$\sigma$-homomorphism, 
 if $\phi$ commutes with $\delta$ and $\sigma$, i.e., $\phi \delta=\delta\phi$ and $\phi \sigma=\sigma\phi$.
 \item For a commutative ring $R$, the ideal generated by $F \subset R$ in $R$ is denoted by $\langle F\rangle$.\hypertarget{diffideal}{}
 \item For a $\delta$-ring $R$, 
 the differential ideal generated by $F\subset R$ in $R$ is denoted by 
 $\langle F\rangle^{(\infty)}$; for a non-negative integer $B$, the ideal in $R$ generated by the set $\{\delta^i(F)\mid 0\leqslant i\leqslant B\}$ in $R$ is denoted by $\langle F\rangle^{(B)}$.
 \end{itemize}
 \end{definition}
  
  \begin{definition}[Differential-difference polynomials]
  \begin{itemize}
  \item[]
  \item\hypertarget{kinfty}{}
 Let $\mathcal{R}$ be a $\delta$-$\sigma$-ring.  
 The differential-difference polynomial ring over $\mathcal{R}$ in $\bm{y}=y_1,\ldots,y_n$, 
 denoted by 
 $\mathcal{R}[\bm{y}_\infty]$
 , is the $\delta$-$\sigma$ ring \[\big(\mathcal{R}[\delta^i\sigma^jy_k\mid i,j\geqslant0;\,1\leqslant k\leqslant n],\delta,\sigma\big),\quad \sigma(\delta^i\sigma^jy_k):=\delta^i\sigma^{j+1}y_k, \ \ \delta(\delta^i\sigma^jy_k):=\delta^{i+1}\sigma^jy_k.\]
 A $\delta$-$\sigma$ polynomial is an element of $\mathcal{R}[\bm{y}_\infty]$. 
 \item \hypertarget{kB}{}
 Given $B\in\mathbb N$, let $\mathcal{R}[\bm{y}_B]$ denote the polynomial ring $\mathcal{R}[\delta^i\sigma^jy_k\mid 0\leqslant i,j\leqslant  B; 1\leqslant k\leqslant n]$.  
 \item
 Given $f\in\mathcal{R}[\bm{y}_\infty]$, the order of $f$ is defined to be the maximal $i+j$ such that $\delta^i\sigma^jy_k$ effectively appears in $f$ for some $k$, denoted by $\ord(f)$.\hypertarget{ordsigmadelta}{}
 \item
 The relative order  of $f$ with respect to $\delta$ (resp. $\sigma$), denoted by $\ord_\delta(f)$ (resp. $\ord_\sigma(f)$),  is defined as the maximal $i$ (resp. $j$) such that 
 $\delta^i\sigma^jy_k$ effectively appears in $f$ for some $k$.
 \item Let $\mathcal R$ be a $\delta$-$\sigma$-ring containing a $\delta$-$\sigma$-field $k$.
Given a point $\bm{a}=(a_1,\dots,a_n)\in \mathcal R^n$, there exists a unique $\delta$-$\sigma$-homomorphism over $k$, 
\[\phi_{\bm{a}}: k[\bm{y}_\infty]\longrightarrow\mathcal R \quad \text{with}\ \ \phi_{\bm{a}}(y_i)=a_i\  \text{and}\ \phi_{\bm{a}}|_k=\Id.\]
 Given $f\in k[\bm{y}_\infty]$, $\bm{a}$ is called  a solution of $f$ in $\mathcal R$ if $f\in\Ker(\phi_{\bm{a}})$.
\end{itemize}
\end{definition}

\begin{definition}[Sequence rings and solutions]\label{def:sequences}
  For a $\delta$-$\sigma$-$k$-algebra $\mathcal{R}$ and $I= \mathbb{N}$ or  $\mathbb{Z}$, the sequence ring $\mathcal{R}^I$  has  the following structure of a $\delta$-$\sigma$-ring ($\delta$-$\sigma^\ast$-ring for $I=\mathbb{Z}$) with $\sigma$ and $\delta$ defined by 
 \[\sigma\big((x_i)_{i\in I}\big):=(x_{i+1})_{i\in I}   \quad\text{and}\quad \delta\big((x_i)_{i\in I}\big):=(\delta(x_{i}))_{i\in I}.\] 
 For a $k$-$\delta$-$\sigma$-algebra $\mathcal{R}$, $\mathcal{R}^I$ can be considered a $k$-$\delta$-$\sigma$-algebra 
  by embedding $k$ into $\mathcal{R}^I$  in  the following way: 
  \[a\mapsto (\sigma^i(a))_{i\in I},\ \  a \in k.\]
For $f\in k[\bm{y}_\infty]$, a solution of $f$ with components in $\mathcal{R}^I$ is called a {\em sequence solution of $f$ in $\mathcal{R}$.} 
\end{definition}
 
  \section{Main result}
  \label{sec:main-result}

   \begin{theorem}[Effective elimination]\label{thm:main}
For all  non-negative integers $r,s$, there exists a computable $B = B(r,s)$ 
  such that, for all:
  \begin{itemize}
  \item non-negative integers $q$,
   \item $\delta$-$\sigma$-fields
    $k$ with $\Char k =0$, and  
\item sets of $\delta$-$\sigma$-polynomials $F\subset k[\hyperlink{kB}{\bm{x}_s
,\bm{y}_s}]$, where $\bm{x} =x_1,\ldots,x_q$, $\bm{y}=y_1,\ldots,y_r$, and $\deg_{\bm{y}}F\leqslant s$,
   \end{itemize}
   we have 
      \[\big\langle \sigma^i(F)\mid i\in \mathbb{Z}_{\geqslant 0} \big\rangle^{\hyperlink{diffideal}{(\infty)}}\cap k[\hyperlink{kinfty}{\bm{x}_{\infty}}]=\{0\} \iff \langle \sigma^i(F)\mid i\in [0,B] \big\rangle^{\hyperlink{diffideal}{(B)}}\cap k[\bm{x}_B]=\{0\}. \]
 \end{theorem}

 By setting $q=0$ in Theorem~\ref{thm:main} and using Proposition~\ref{prop:strongnullstellensazt}, we obtain:
 
 \begin{cor}[Effective Nullstellensatz]For all  non-negative integers $r,s$, there exists a computable $B = B(r,s)$ 
  such that, for all:
  \begin{itemize}
   \item $\delta$-$\sigma$-fields
    $k$ with $\Char k =0$, and   
\item  sets of $\delta$-$\sigma$-polynomials $F\subset k[\bm{y}_s]$, where $\bm{y}=y_1,\ldots,y_r$ and $\deg_{\bm{y}}F\leqslant s$, 
   \end{itemize}
   the following statements are equivalent:
   \begin{enumerate}
  \item There exists a $\delta$-field $L$ such that $F=0$ has a sequence solution in $L$.
   \item $1 \notin \langle \sigma^i(F)\mid i\in [0,B] \big\rangle^{(B)}$.
   \item  There exists a field $L$ such that the polynomial system
     $\big\{\sigma^i(F)^{(j)}=0\mid i,j \in [0,B]\big\}$
      in the finitely many unknowns $\bm{x}_{B+s}$ has a solution in  $L$.
   \end{enumerate}
 \end{cor}

\section{Definitions and notation used in the proofs}
\label{sec:defproof}

Let $\mathcal{R}$ be a $\delta$-$\sigma$-ring.
\begin{itemize}
\hypertarget{Rrs}
\item For $r,s\in\mathbb N$, let $\mathcal{R}[\bm{y}_{r,s}]$ and $\mathcal{R}[\bm{y}_{\infty,s}]$ denote the polynomial ring \[\mathcal{R}[\delta^i\sigma^jy_k\mid 0\leqslant i\leqslant r, 0\leqslant j \leqslant s; 1\leqslant k\leqslant n]\] and the $\delta$-ring \[\mathcal{R}[\delta^i\sigma^jy_k\mid i\geqslant 0,\, 0\leqslant j \leqslant s;\, 1\leqslant k\leqslant n],\] respectively. Additionally, $\mathcal{R}(\bm{y}_{r,s})$ and $\mathcal{R}(\bm{y}_{\infty,s})$ denote their fields of fractions.
\item The radical of an ideal $I$ in a commutative ring $R$ is denoted by $\sqrt{I}$.\hypertarget{deltrdeg}{}
\item  Let $k\subset L$ be two $\delta$-fields.
A subset $S\subset L$ is said to be $\delta$-independent over $k$, if the set $\{\delta^ks\mid k\geqslant0,\,s\in S\}$ is algebraically independent over $k$.
The cardinality of any maximal subset of $L$ that is 
$\delta$-independent over $k$ is denoted by $\deltrdeg L/k$.
\item In what follows, we will consider every $\delta$-field $(k, \delta)$ as a $\delta$-$\sigma^*$-field with respect to $\delta$ and the identity automorphism.
From this standpoint, the ring of differential polynomials over $k$ in $\bm{y}$ (see~\citep[Chapter~I, \S~6]{Kol}) can be realized as $k[\bm{y}_{\infty, 0}] \subset k[\bm{y}_{\infty}]$. 
We use $k(\bm a_{\infty,0})$ to denote the differential field extension of $k$ generated by a tuple $\bm{a}$.
\end{itemize}

 \begin{definition}
 A $\delta$-field $K$ is called {\em differentially closed} if, for all $F \subset K[\bm{y}_{\infty,0}]$ and $\delta$-fields $L$ containing $K$, the existence of a solution to $F=0$ in $L$ implies the existence of a solution to $F=0$ in $K$.
 \end{definition}
 
\begin{definition}[Differential varieties and $\diffspec$] \label{def:varietyanddiffspec}    Let $(K,\delta)$ be a differentially closed field containing a differential field $(k,\delta)$ and $\bm{y} = y_1,\ldots,y_n$.
 \begin{itemize}
\item For  $F\subset  
k[\hyperlink{Rrs}{\bm{y}_{\infty,0}}]$, we write 
 \[\mathbb{V}(F) = \{\bm{a} \in K^n\mid \forall f\in F\, f(\bm{a})=0\}.\]
     \item A subset $X \subset K^n$ is called a {\em differential variety} over $k$ if there exists $F\subset  k[\bm{y}_{\infty,0}]$ such that $X=\mathbb{V}(F)$.\hypertarget{diffspec}{}
     \item For a subset $X \subset K^n$, we also write $X = \diffspec R$ if there exists 
     $F\subset K[\bm{y}_{\infty,0}]$ 
     such that $X =\mathbb{V}(F)$ and 
     $R = K[\bm{y}_{\infty,0}]/\langle F\rangle^{\hyperlink{diffideal}{(\infty)}}$
     (note that $R$ is not assumed to be reduced).  We define 
     $R_X :=  K[\bm{y}_{\infty,0}]\big/\sqrt{\langle F\rangle^{(\infty)}}$.
     \item A differential variety $\mathbb{V}(F)$ is called {\em irreducible} if $\sqrt{\langle F\rangle^{(\infty)}}$ is a prime ideal.
     \item The {\em generic point} $(a_1,\ldots,a_n)$ of  an irreducible $\delta$-variety $X=\mathbb{V}(F)$ is the image of the $\bm{y}$ under the homomorphism $K[\bm{y}_{\infty,0}]\to  K[\bm{y}_{\infty,0}]\big/\sqrt{\langle F\rangle^{(\infty)}}$.\hypertarget{ddim}{}
     \item Taking differential varieties as the basic closed sets, we define \emph{the Kolchin topology} on $K^n$.\hypertarget{Kol}{}
     \item For a subset $S \subset K^n$, we define \emph{the Kolchin closure} of $S$ (denoted by $\overline{S}^{\kol}$) to be the intersection of all differential subvarieties of $K^n$ containing $S$.  
 \end{itemize}
    Let $X$ be an irreducible $\delta$-variety with the generic point $\bm{a} = (a_1,\ldots,a_n)$.
     \begin{itemize}
         \item The {\em differential dimension} of $X$, denoted by $\delta$-$\dim(X)$, is defined as  
     $\deltrdeg K(\bm{a}_{\infty,0})/K$.
           \item A {\em parametric set} of $X$ is a subset $\{y_{i}\mid i\in I\}\subset \{y_1,\ldots,y_n\}$ such that $\{a_i\mid i\in I\}$ is a differential transcendence basis of $K(\bm{a}_{\infty,0})$ over $K$.\hypertarget{relord}{}
           \item The {\em relative order} of $X$ with respect to a parametric set $U=\{y_{i}\mid i\in I\}$, denoted by $\ord_{U}X$, is defined as 
      \[\ord_U X=\trdeg  K(\bm{a}_{\infty,0})\big/K\big((a_i, i\in I)_{\infty,0}\big).\hypertarget{ord}{}\]
           \item The {\em order} of $X$ is the maximum of all the relative orders of $X$  (\cite[Theorem 2.11]{GLY}), that is,   \[\ord(X)=\max\{\ord_UX\mid \text{$U$ is a parametric set of $X$}\}.\]  
     \end{itemize} 
     
\end{definition}

 \section{What is a
 universal $\delta$-$\sigma$-ring for solving equations?}\label{sec:seqfun}
 
 This section is devoted to answering the question 
 \begin{question}In what rings is it natural to look for solutions of  differential-difference equations?
 \end{question}
 We will show 
 that rings of sequences are universal solution rings in the abstract mathematical sense.
 More precisely, we prove an analogue of the Hilbert Nullstellensatz, Proposition~\ref{prop:strongnullstellensazt}.
 On the other hand, from the applications standpoint, it would be natural if solutions of delay-differential equations were functions defined on a 
 subset of the complex plane or real line.
 It turns out that these two seemingly contradictory standpoints can be viewed as closely related via the construction described below.
 
 \begin{definition}[Rings of meromorphic functions]
 \begin{itemize}[itemsep=0in]
 \item[]\hypertarget{MU}{}
     \item Let $U \subset \mathbb{C}$ be an open nonempty set. We denote the ring of meromorphic functions on $U$ by $\mathcal{M}(U)$.
     $\mathcal{M}(U)$ is a field if and only if $U$ is connected.
     \hypertarget{MD}{}
     \item Let $D \subset \mathbb{C}$ be a nonempty discrete set. 
     We define a ring $\mathcal{M}(D)$ of germs of meromorphic functions on $D$
     as the quotient 
     \[
       \mathcal{M}(D) := \{ (f, U) \mid U \text{ is open such that }D \subset U,\; f \in \mathcal{M}(U)  \}/\sim,
     \]
     where the equivalence relation $\sim$ is defined by
     \[
       \bigl( (f_1, U_1) \sim (f_2, U_2) \bigr) \iff (\forall z \in U_1 \cap U_2, \; f_1(z) = f_2(z)).
     \]
     
     \item For every open nonempty $U \subset \mathbb{C}$, $\mathcal{M}(U)$ is a $\delta$-ring with respect to the standard derivative.
     If $U = U + \{1\}$, 
     then $\mathcal{M}(U)$ can be considered as a $\delta$-$\sigma^\ast$-ring with respect to the shift automorphism 
  $\sigma(f)(z) = f(z-1)$. 
     Similarly, for a nonempty discrete $D \subset \mathbb{C}$, $\mathcal{M}(D)$ is a $\delta$-ring and.
     If additionally $D = D + \{1\}$, then $\mathcal{M}(D)$ is a $\delta$-$\sigma^\ast$-ring with $\sigma$ sending the equivalence class of $(f(z), U)$ to the equivalence class of $(f(z - 1), U + \{1\})$. 
 \end{itemize}
 \end{definition}
 
 \begin{definition}[Transforms between functions and sequences]\label{def:seq_to_fun}
 {\color{white}Dirty hack}
 \begin{itemize}[itemsep=0in]
     \item We define $S := \{z \in \mathbb{C} \mid -0.5 < \operatorname{Re} z < 0.5\} \subset \mathbb{C}$.
     
     \item Consider a nonempty  open subset $U \subset \mathbb{C}$ such that $U = U + \{1\}$.
     Then we define a map \[
     \varphi_U \colon \mathcal{M}(U) \to (\mathcal{M}(U\cap S))^{\mathbb{Z}}
     \]
     as follows.
     For every $f \in \mathcal{M}(U)$ and every $j \in \mathbb{Z}$, we define $f_j \in \mathcal{M}(U \cap S)$ by $f_j(z) := f(z + j)$.
     Then we set $\varphi_U(f) := (\ldots, f_{-1}, f_0, f_1, \ldots)$.
     One can check that $\varphi_U$ defines an injective  homomorphism of $\delta$-$\sigma^\ast$-rings, where $(\mathcal{M}(U \cap S))^{\mathbb{Z}}$ bears a $\delta$-$\sigma^\ast$-ring structure as descibed in Definition~\ref{def:sequences}.
     The same can be done for a nonempty discrete $D \subset \mathbb{C}$ such that $D = D + \{1\}$ and $D \cap \partial S = \varnothing$.
     
     \item Consider a nonempty open subset $U_0 \subset S$. We define a map 
     \[
     \psi_{U_0}\colon (\mathcal{M}(U_0))^{\mathbb{Z}} \to \mathcal{M}(U_0 + \mathbb{Z})
     \]
     as follows.
     For every $\{f_j\}_{j \in \mathbb{Z}} \in (\mathcal{M}(U_0))^{\mathbb{Z}}$, we define a function $f \in \mathcal{M}(U_0 + \mathbb{Z})$ by setting $f(z)|_{
     U_0
     + \{j\}} := f_j(z - j)$  for every $j \in \mathbb{Z}$.
     Then we define $\psi_{U_0}(\{f_j\}_{j \in \mathbb{Z}}) := f$.
     One can check that $\psi_{U_0}$ defines an isomorphism of $\delta$-$\sigma^\ast$-rings.
     The same can be done for a nonempty discrete $D \subset S$.
 \end{itemize}
 \end{definition}

In Section~\ref{subsec:germs}, we show that  $\mathcal{M}(\mathbb{Z})$ is a universal solution ring for $\delta$-$\sigma$-equations over $\mathbb{C}$ (Proposition~\ref{prop:strongnullstellensazt_germs}) and derive a version of our effective elimination theorem for this case (Corollary~\ref{cor:elim_meromorphic}).
 Moreover, in Section~\ref{subsec:not_embedding}, we show that there exists a system of $\delta$-$\sigma$-equations that
 \begin{itemize}[itemsep=0in]
 \item has a solution in $\mathcal{M}(\mathbb{Z})$ but,
 \item for every open $U \subset \mathbb{C}$ such that $U = U + \{1\}$, does not have a solution in $\mathcal{M}(U)$.
 \end{itemize}
 %%%%%%%%%%%%%%%%%%%%%%%%%%%%%%%%%%%%%%%
 
 \subsection{Solutions in sequences over $\delta$-fields}
 \label{subsec:sequences}
 
   \begin{prop}\label{prop:strongnullstellensazt}
  Let $n \in \mathbb{Z}_{\geqslant 0}$, $k$ be a $\delta$-$\sigma^\ast$-field.
  Then, for every $F \subset k[\hyperlink{kinfty}{\bm{y}_\infty}]$ and $f \in k[\bm{y}_\infty]$ with $\bm{y} = y_1, \ldots, y_n$,
  the following statements are equivalent:
  \begin{enumerate}[label=(\arabic*)]
     \item\label{item1:seq} for every $\delta$-$\sigma$-field extension $k \subset K$, $f$ vanishes on all solutions of $F=0$ in $K^\mathbb Z$.
  \item\label{item2:ideal} There exists $m\in\mathbb N$ such that 
  \[
  \sigma^m(f^m) \in \langle\sigma^j(F) \;|\; j\in \mathbb{Z}_{\geqslant 0}\rangle^{\hyperlink{diffideal}{(\infty)}} \subset k[\bm{y}_\infty].\]
  \end{enumerate}
  Moreover, if 
  $\sigma = \operatorname{id}_k$, then~\ref{item2:ideal} is equivalent to: for every $\delta$-$\sigma$-field extension $k \subset K$ with $\sigma = \operatorname{id}_K$, $f$ vanishes on all the sequence solutions of $F$ in $K^\mathbb Z$.
  \end{prop}
 \begin{proof}
The implication \ref{item2:ideal}~$\implies$~\ref{item1:seq} is straightforward because $\sigma$ is injective.
 It remains to show  \ref{item1:seq}~$\implies$~\ref{item2:ideal}. 
 Suppose that~\ref{item2:ideal} does not hold. Let
 \[
   \mathcal{I} := \sqrt{\langle\sigma^j(F) \;|\; j\in \mathbb{Z}\rangle^{(\infty)}} \subset k[\sigma^j(\hyperlink{Rrs}{\bm{y}_{\infty, 0}}) \mid j \in \mathbb{Z}].
 \]
By \cite[Theorem 2.1]{Kaplansky}, $\mathcal I$ is an intersection of prime $\delta$-ideals (maybe, an infinite intersection).
Assume that $f \in \mathcal{I}$.
Then there exists $m \in \mathbb{N}$ such that 
\[
f^m \in \langle \sigma^{j}(F) \mid j \in [-m, m] \rangle^{\hyperlink{diffideal}{(\infty)}}.
\]
Applying $\sigma^m$, we have $\sigma^m(f^m) \in \langle\sigma^j(F) \;|\; j\in \mathbb{Z}_{\geqslant 0}\rangle^{(\infty)}$, and this contradicts to the assumption that~\ref{item2:ideal} does not hold.
Thus, $f\notin \mathcal I$, so there exists a prime $\delta$-ideal $P$ with $\mathcal I\subseteq P$ and $f\notin P$.
 Let $\mathcal{U}_0$ be the quotient field of the $\delta$-domain $k[\sigma^j(\bm{y}_{\infty, 0}) \mid j \in \mathbb{Z}] / P$ that has a natural structure of $\delta$-field.
 Let $\mathcal{U}$ be a differentially closed field containing $\mathcal{U}_0$.
 \citep[Lemma~2.3]{MarkerModel} together with Zorn's lemma implies that $\sigma$ can be extended from $k$ to $\mathcal{U}$ so that $\mathcal{U}$ is a $\delta$-$\sigma$-field. 
 Note that if $\sigma|_k = \operatorname{id}$, then we can set $\sigma|_{\mathcal{U}} = \operatorname{id}$.
  Let 
  \[
  \eta=\big((\overline{\sigma^jy_1})_{j\in\mathbb Z},\ldots,(\overline{\sigma^jy_n})_{j
  \in \mathbb Z}\big)\in \mathcal{U}^{\mathbb Z}\times\cdots\times \mathcal{U}^{\mathbb Z},
  \]
  where $\overline{\sigma^jy_k}$ is the canonical image of $\sigma^jy_k$.
   Clearly, $\eta$ is a solution of $F = 0$ in $\mathcal{U}^\mathbb Z$
   but $f$ does not vanish at it.
   Thus, \ref{item1:seq} does not hold. So, \ref{item1:seq} implies \ref{item2:ideal}.
 \end{proof}

 \begin{remark}\label{rem:nullst}
   The proof of Proposition~\ref{prop:strongnullstellensazt} can be modified to show that the following conditions are also equivalent
   \begin{enumerate}[label=(\arabic*)]
     \item\label{item:seq} for every $\delta$-$\sigma$-field extension $k \subset K$, $f$ vanishes on all solutions of $F = 0$ in $K^{\mathbb{N}}$.
    
  \item\label{item:ideal} There exists $m\in\mathbb \mathbb{Z}_{\geqslant 0}$ such that
  \[
  f^m \in \langle\sigma^j(F) \;|\; j\in \mathbb{Z}_{\geqslant 0}\rangle^{(\infty)} \subset k[\hyperlink{kinfty}{\bm{y}_\infty}].
  \]
   \end{enumerate}
 \end{remark}
 
 \begin{remark} \label{rem:nullst2}
   In the case $f = 1$ (so-called weak Nullstellensatz), the second condition of Proposition~\ref{prop:strongnullstellensazt} is equivalent to the second condition in Remark~\ref{rem:nullst}.
   Thus, for $f = 1$, all the conditions of Proposition~\ref{prop:strongnullstellensazt} and Remark~\ref{rem:nullst} are equivalent.
   However, they are not equivalent for general $f$ as the following example shows.
 \end{remark}
 
\begin{example} 
  Let $k = \mathbb{Q}$. Consider
  \[
  F = \{y^2 - \sigma(y), y^2 - \sigma^2(y)\} \quad \text{ and } \quad f = y(y - 1).
  \]
  Let $k \subset K$ be an extension of $\delta$-$\sigma$-fields and $\overline{a} = 
  (\ldots, a_{-1}, a_0, a_1, a_2, \ldots) \in K^{\mathbb{Z}}$ a solution of $F$.
  For every $i \in \mathbb{Z}$, we have
  \[
  a_{i - 1}^2 - a_i = a_{i - 1}^2 - a_{i + 1} = 0 \implies a_i = a_{i + 1}.
  \]
  Combining with $a_i^2 = a_{i + 1}$, we have $a_i^2 = a_i$.
  Thus, $f$ vanishes at $\overline{a}$.
  However,  $f$ does not vanish on the solution $(-1, 1, 1,\ldots)$ of $F=0$ in $\mathbb{Q}^{\mathbb{N}}$.
\end{example}

%%%%%%%%%%%%%%%

\subsection{Solutions in germs}
\label{subsec:germs}

\begin{prop}\label{prop:strongnullstellensazt_germs}
 For every $n \in \mathbb{Z}_{\geqslant 0}$, $F \subset \mathbb{C}[\hyperlink{kinfty}{\bm{y}_\infty}]$, and $f \in \mathbb{C}[\bm{y}_\infty]$ with $\bm{y} = y_1, \ldots, y_n$,
  the following statements are equivalent:
  \begin{enumerate}[label=(\arabic*),itemsep=0in]
     \item\label{item:seq_germ} $f$ vanishes on all the solutions of $F = 0$ in $\hyperlink{MD}{\mathcal{M}(\mathbb{Z})}$.
  \item\label{item:ideal_germ} There exists $m\in\mathbb N$ such that 
  \[
  \sigma^m(f^m) \in \langle\sigma^j(F) \;|\; j\in \mathbb{Z}_{\geqslant 0}\rangle^{\hyperlink{diffideal}{(\infty)}} \subset \mathbb{C}[\bm{y}_\infty].\]
  \end{enumerate}
  \end{prop}
  
  \begin{proof}
    The implication \ref{item:ideal_germ}~$\implies$~\ref{item:seq_germ} is straightforward.
 It remains to show  \ref{item:seq_germ}~$\implies$~\ref{item:ideal_germ}. 
 Suppose that~\ref{item:ideal_germ} does not hold. 
 Let $E$ be the subfield of $\mathbb{C}$ generated by the coefficients of $F$ and $f$ over  $\mathbb Q$.
 Proposition~\ref{prop:strongnullstellensazt} implies that there exists a $\delta$-field $K \supset E$ such that $F = 0$ has a solution $\overline{a} = \{a_j\}_{j \in \mathbb{Z}}$ in $K^{\mathbb{Z}}$ such that $f(\overline{a}) \neq 0$.
 Replacing $K$ by its $\delta$-subfield generated by $E$ and $\{a_j\}_{j \in \mathbb{Z}}$, we can further assume that $K$ is an at most countably generated $\delta$-field extension of $E$.
 Hence $K$ is at most countable.
 \citep[Lemma~A.1]{MarkerModel} implies that there exists a homomorphism of $\delta$-fields $\theta \colon K \to \mathcal{M}(0)$ that maps $E \subset K$ isomorphically to $E \subset \mathbb{C} \subset \mathcal{M}(0)$.
 This homomorphism can be extended to an injective homomorphism $\theta \colon K^{\mathbb{Z}} \to (\mathcal{M}(0))^{\mathbb{Z}}$ of $\delta$-$\sigma^\ast$-algebras over $E$.
 Then the composition of $\theta$ with the isomorphism $\psi_0 \colon (\mathcal{M}(0))^{\mathbb{Z}} \to \mathcal{M}(\mathbb{Z})$ (see Definition~\ref{def:seq_to_fun}) is an injective homomorphism of $\delta$-$\sigma^\ast$ algebras over $E$.
 
 We set $b := \psi_0\circ \theta (\overline{a}) \in \mathcal{M}(\mathbb{Z})$.
 Then $b$ is a solution of $F = 0$ and, since $\psi_0\circ \theta$ is injective, $f$ does not vanish at $b$.
 This contradicts~\ref{item:seq_germ}.
  \end{proof}

Combining Proposition~\ref{prop:strongnullstellensazt_germs} with Theorem~\ref{thm:main}, we obtain:

\begin{cor}\label{cor:elim_meromorphic}
  For all  non-negative integers $r,s$, there exists a computable $B = B(r,s)$ 
  such that, for all:
  \begin{itemize}
  \item non-negative integer $q$,
  \item a set of $\delta$-$\sigma$-polynomials $F\subset \mathbb{C}[\hyperlink{kinfty}{\bm{x}_\infty,\bm{y}_s}]$, where $\bm{x} =x_1,\ldots,x_q$, $\bm{y}=y_1,\ldots,y_r$, and $\deg_{\bm{y}}F\leqslant s$,
   \end{itemize}
   the following statements are equivalent
   \begin{itemize}
       \item there exists a nonzero $g \in \mathbb{C}[\bm{x}_{\infty}]$ that vanishes on every solution of $F = 0$ in $\mathcal{M}(\mathbb{Z})$;
       \item $\langle \sigma^i(F)\mid i\in [0,B] \big\rangle^{\hyperlink{diffideal}{(B)}}\cap \mathbb{C}[\bm{x}_{\infty}] \neq \{0\}$.
   \end{itemize}
 \end{cor}

%%%%%%%%%%%%%%%%

\subsection{Solutions in meromorphic functions on open subsets of $\mathbb{C}$}
\label{subsec:not_embedding}

In this section, we will present a specific system of $\delta$-$\sigma$-equations~\eqref{eq:main} that has a solution in~$\mathcal{M}(\mathbb{Z})$ but, for any open $U \subset \mathbb{C}$ such that $U = U + \{1\}$, does not have a solution in $\hyperlink{MU}{\mathcal{M}(U)}$  (see Proposition~\ref{prop:solutions_in_domain}). 
We recall some relevant facts about the Weierstrass $\wp$-function:
\begin{itemize}
    \item Let $g_2, g_3 \in \mathbb{C}$ be the complex numbers such that the Weierstrass function $\wp(z)$ with periods $1$ and $i$ (the imaginary unit) is a solution of 
    \begin{equation}\label{eq:wp}
      (x')^2 = 4x^3 - g_2x - g_3.
    \end{equation}
    We will use the fact that every nonconstant solution of~\eqref{eq:wp} is of the form $\wp(z + z_0)$ for some $z_0 \in \mathbb{C}$, see~\cite[page~39, Korollar~F]{KK}.
    
    \item Recall that the field of doubly periodic  meromorphic functions on $\mathbb{C}$ with periods $1$ and $i$ is generated by $\wp(z)$ and $\wp'(z)$~\cite[page~8, Theorem~4]{Lang}.
    Let $\omega := 1 + \sqrt{2}i$, and consider a rational function $R(x_1, x_2) \in \mathbb{C}(x_1, x_2)$ such that
    \begin{equation}\label{eq:defR}
    \wp(z + \omega) = R(\wp(z), \wp'(z)).
    \end{equation}
\end{itemize}

%%%%%%%%%

\begin{prop}\label{prop:solutions_in_domain}
Consider the following system of algebraic differential-difference equations in the unknowns $x$, $y$, $w$:
\begin{equation}\label{eq:main}
\begin{cases}
  (x')^2 = 4x^3 - g_2 x - g_3,\\
  \sigma(x) = R(x, x'),\\
  y^3 = \frac{1}{x},\\
  x' w = 1.
\end{cases}
\end{equation}
  \begin{enumerate}[label=(\arabic*)]
      \item\label{item:exists_solution} System~\eqref{eq:main} has a solution in $\hyperlink{MZ}{\mathcal{M}(\mathbb{Z})}$.
      \item\label{item:no_solutions} For every nonempty open subset $U \subset \mathbb{C}$ such that $U = U + \{1\}$, system~\eqref{eq:main} does not have a solution in $\hyperlink{MU}{\mathcal{M}(U)}$. 
  \end{enumerate}
\end{prop}

\begin{proof}
  \underline{Proof of~\ref{item:exists_solution}.}
  Let $K$ be the algebraic closure of the field
  $\mathcal{M}(\mathbb{C})$. 
  We set 
  \[
    x_j = \wp(z + j\omega),  \quad y_j = \sqrt[3]{\frac{1}{\wp(z + j\omega)}}, \quad \text{ and } \quad w_j = \frac{1}{\wp'(z + j\omega)}.
  \]
The first equation in~\eqref{eq:main} holds for these sequences because every shift of $\wp(z)$ is its solution being an equation with constant coefficients.
The second equation in~\eqref{eq:main} holds because
\[
  x_{j + 1} = \wp(z + (j + 1)\omega) = R(\wp(z + j\omega), \wp'(z + j\omega)) = R(x_j, x_j')
\]
due to~\eqref{eq:defR}.
A direct computation shows that the last two equations in~\eqref{eq:main} also hold.
Thus, the system~\eqref{eq:main} has a solution in $K^{\mathbb{Z}}$.
Combining Propositions~\ref{prop:strongnullstellensazt} and~\ref{prop:strongnullstellensazt_germs}, we see that~\eqref{eq:main} has a solution in $\mathcal{M}(\mathbb{Z})$.

\underline{Proof of~\ref{item:no_solutions}.}
Assume the contrary, let $U \subset \mathbb{C}$ be such a subset and $(x(z), y(z),w(z))$ be such a solution.
  Since~\eqref{eq:main} is autonomous, we can assume that $0 \in U$ by shifting $U$ and the solution if necessary.
  We denote the connected component of $U$ containing $0$ by $U_0$.
  The last equation of~\eqref{eq:main} implies that $x(z)|_{U_0}$ is nonconstant.
  Then the first equation of~\eqref{eq:main} implies that there exists $z_0 \in \mathbb{C}$ such that 
  \begin{equation}\label{eq:xU0}
x(z)|_{D_0} = \wp(z + z_0).
  \end{equation}
  We will prove that, for every $z \in U_0$ and $s \in \mathbb{Z}_{\geqslant 0}$,
  \begin{equation}\label{eq:shift}
    x(z + s) = \wp(z + z_0 + s\omega)
  \end{equation}
  by induction on $s$.
  The base case $s = 0$ follows from~\eqref{eq:xU0}.
  Assume that~\eqref{eq:shift} holds for $s \geqslant 0$.
  Then, using the second equation in~\eqref{eq:main}, the inductive hypothesis, and~\eqref{eq:defR}, we have
  \[
  x(z + s + 1) = R(x(z + s), x'(z + s)) = R(\wp(z + z_0 + s\omega), \wp'(z + z_0 + s\omega)) = \wp(z + z_0 + (s + 1)\omega).
  \]
  This proves~\eqref{eq:shift}.

  Let $\varepsilon > 0$ be a real number such that $U$ contains the $\varepsilon$-neighbourhood of $0$.
  Kronecker's theorem implies that there exist $s \in \mathbb{Z}_{\geqslant 0}$ and $m, n \in \mathbb{Z}$ (which we fix) such that 
  \[
  |z_0 + s\omega - n - m i| < \varepsilon.
  \]
  We set $z_1 = n + mi - z_0 - s\omega$.
  Then, since $|z_1| < \varepsilon$, we have $z_1 \in U_0$, and so
  $z_1 + s \in U$. 
  \eqref{eq:shift} implies
  \[
  x(z_1 + s) = \wp(z_1 + z_0 + s\omega) = \wp(n + m i) = \infty.
  \]
  Then $y(z_1 + s) = 0$. 
  Let $d \geqslant 1$ be the order of zero of $y$ at $z_1 + s$.
  Then $x = \frac{1}{y^3}$ has a pole of order $3d$ at $z_1 + s$.
  We arrive at a contradiction with the fact that all the poles of $\wp$ are of order two \cite[page~8]{Lang}.
\end{proof}

%%%%%%%%%%%%%%%%%%%%%%%%%%%%%%%%%%%%%%%%%%

\section{Proof of the main result} 
\label{sec:proofmain}

The proofs are structured as follows.
In Section~\ref{subsec:big_field}, we embed the ground $\delta$-$\sigma$-field $k$ to a differentially closed $\delta$-$\sigma^\ast$-field $K$.
In Section~\ref{subsec:trains}, we extend the technique of trains (developed in~\cite{OPS} for difference equations) to the differential-difference case.
Section~\ref{subsec:bounds} begins with Section~\ref{sec:numberofcomponents}, in which we establish a bound (Corollary~\ref{cor:numberofcomponents}) for the number of components of a differential-algebraic variety.
This bound replaces the B\'ezout bound, extensively used in~\cite{OPS} but lacking in the differential-algebraic setting.
In Section~\ref{sec:bound_trains}, we show that the existence of a sufficiently long train implies the existence of a solution in $K^{\mathbb{Z}}$.
Finally, in Section~\ref{sec:main_proof}, we use these ingredients to prove the main result, Theorem~\ref{thm:main}.

\subsection{Constructing big enough field $K \supset k$}\label{subsec:big_field}

Throughout Section~\ref{sec:proofmain}
\begin{itemize}
    \item $k$ is the $\delta$-$\sigma$-field from Theorem~\ref{thm:main},
    \item $K$ is a fixed differentially closed  $\delta$-$\sigma^\ast$-field containing $k$. 
    The existence of such field follows from Lemma~\ref{lem:embedding}.
\end{itemize}

\begin{lemma}\label{lem:embedding}
 For every $\delta$-$\sigma$ field $k$ of characteristic zero, there exists an extension $k \subset K$ of $\delta$-$\sigma$-fields, where $K$ is a differentially closed $\delta$-$\sigma^*$-field.
\end{lemma}
\begin{proof}
We will show that there exists a $\delta$-$\sigma^\ast$ field $K_0$ containing $k$.
  The proof of~\citep[Proposition~2.1.7]{LevinBook} implies that one can build an ascending chain of $\sigma$-fields
  \begin{equation}\label{eq:asc_chain}
  k_0 \subset k_1 \subset k_2 \subset \ldots
  \end{equation}
  such that, for every $i \in \mathbb{N}$, there exists an isomorphism $\varphi_i\colon k \to k_i$ of $\sigma$-fields,  $\sigma(k_{i + 1}) = k_i$, and $\varphi_i = \sigma\circ\varphi_{i + 1}$ for every $i \in \mathbb{N}$.
  We transfer the $\delta$-$\sigma$-structure from $k$ to $k_i$'s via $\varphi_i$'s.
  Then $\varphi_i = \sigma\circ \varphi_{i + 1}$ implies that the restriction of $\delta$ on $k_{i + 1}$ to $k_i$ coincides with the action of $\delta$ on $k_i$.
  We set $K_0 := \bigcup\limits_{i \in \mathbb{N}} k_i$.
  Since the action $\delta$ and $\sigma$ is consistent with the ascending chain~\eqref{eq:asc_chain}, $K_0$ is a $\delta$-$\sigma$-extension of $k_0 \cong k$.
  It is shown in~\cite[Proposition~2.1.7]{LevinBook} that the action of $\sigma$ on $K_0$ is surjective. 
  \cite[Theorem~3.15]{BM07} implies that $K_1$ can be embedded in a differentially closed $\delta$-$\sigma^\ast$-field $K$.
\end{proof}

\subsection{Partial solutions and trains}\label{subsec:trains}

  \begin{definition} \label{def-partialsolution}
   Let $F \subset k[\hyperlink{kinfty}{\bm{y}_\infty}]$, where $\bm{y}= y_1,\ldots,y_n$, be a set of $\delta$-$\sigma$-polynomials.
   Suppose $h=\max\{\hyperlink{ordsigmadelta}{\ord_\sigma}(f) \mid f \in F\}$. 
   A sequence  of tuples $(\overline{a}_1,\ldots,\overline{a}_n)\in K^{\ell+h}\times\cdots\times K^{\ell+h}$ is called a {\em partial solution} of $F$ of length $\ell$ if
   $(\overline{a}_1,\ldots,\overline{a}_n)$ is a $\delta$-solution of the system in $\bm{y}_{\infty,\ell+h-1}$:
   \[
   \{\sigma^i (F)=0 \;|\; 0\leqslant i\leqslant\ell-1\}.
   \]
  \end{definition}

  \begin{example}
  Let $k=\mathbb Q(t)$ be the $\delta$-$\sigma^\ast$-field with $\delta=\frac{d}{dt}$ and $\sigma(f(t))=f(t+1)$ for each $f(t)\in\mathbb Q (t)$.
  Let $f=t\cdot\delta(y)+\sigma(y)\in k[y_\infty]$. So $h=1$.
  A partial solution of $f$ of length $\ell$ is a sequence $(a_0,a_1,\ldots,a_{\ell})\in K^{\ell+1}$  such that
\[  (t+i)\cdot\delta(a_i)+a_{i+1}=0, \,i=0,1,\ldots,\ell-1.\]
A solution of $f$ in $K^{\mathbb N}$ is a sequence $(a_i)_{i\in \mathbb N}\in K^{\mathbb N}$
such that for each $i\in\mathbb N$, 
\[  (t+i)\cdot\delta(a_i)+a_{i+1}=0.\]
  \end{example}

  With the above set $F$ of $\delta$-$\sigma$-polynomials, we associate the following geometric data analogously to 
  \cite{OPS}:
  \begin{itemize}
  \item the $\delta$-variety $X\subset\mathbb A^H$ defined by $f_1=0,\ldots,f_N=0$ regarded as $\delta$-equations in 
 $k[\hyperlink{Rrs}{\bm{y}_{\infty,h}}]$
  with $H=n(h+1)$, and so \[X=\hyperlink{diffspec}{\diffspec} R_X,\quad R_X := 
  K[\bm{y}_{\infty,h}]\big/\sqrt{(f_1,\ldots,f_N)^{(\infty)}}
  ;\]
 \item two 
 projections $\pi_1,\pi_2:\mathbb A^H\longrightarrow \mathbb A^{H-n}$ defined by 
 \begin{align*}\pi_1(a_1,\ldots,\sigma^h(a_1);\ldots; a_n,\ldots,\sigma^h(a_n))&:=(a_1,\sigma(a_1),\ldots,\sigma^{h-1}(a_1);\ldots; a_n,\ldots,\sigma^{h-1}(a_n)),\\
  \pi_2(a_1,\ldots,\sigma^h(a_1);\ldots; a_n,\ldots,\sigma^h(a_n))&:=(\sigma(a_1),\ldots,\sigma^{h}(a_1);\ldots; \sigma(a_n),\ldots,\sigma^{h}(a_n)).
  \end{align*}
  \end{itemize}
  
  By $\sigma(X)$, we mean the $\delta$-variety in $\mathbb A^H$ defined by $f_1^{\sigma},\ldots,f_N^{\sigma}$, 
  where $f_i^{\sigma}$ is the result by applying $\sigma$ to the coefficients of $f_i$.

 \begin{definition}
 A sequence $p_1,\ldots,p_\ell\in\mathbb A^H$ is a {\em partial solution of the triple} $(X,\pi_1,\pi_2)$  if 
\begin{enumerate}[label=\arabic*)]
\item  for all $i$, $1\leqslant i \leqslant \ell$, we have $p_i\in\sigma^{i-1}(X)$ and
\item for all $i$, $1\leqslant i<\ell$, we have $\pi_1(p_{i+1})=\pi_2(p_i)$.
\end{enumerate}
 A two-sided infinite sequence with such a property is called a {\em solution of the triple} $(X,\pi_1,\pi_2)$.
 \end{definition}

 \begin{lemma}\label{lem:64}
 For every positive integer $\ell$, $F$ has a partial solution of length $\ell$ if and only if the triple $(X,\pi_1,\pi_2)$ has a partial solution of length $\ell.$
 System $F$ has a solution in $K^\mathbb Z$ if and only if the triple $(X,\pi_1,\pi_2)$ has a   solution.
 \end{lemma}
 \begin{proof}
  It suffices to show that the first assertion holds.
 Suppose that $(a_0,a_1,\ldots,a_{h+\ell-1})$ is a partial solution of $F$.
 Let $p_i=(a_{i-1},a_i,\ldots,a_{i-1+h})$ for $i=1,\ldots,\ell$.
 Then $p_1,\ldots,p_\ell$ is a partial solution of $(X,\pi_1,\pi_2)$ of length $\ell$.
 For the other direction, let $p_1,\ldots,p_\ell$ be a partial solution of $(X,\pi_1,\pi_2)$ of length $\ell$. 
 Since  $\pi_1(p_{i+1})=\pi_2(p_i)$, there exist { $a_0,\ldots,a_{h+\ell-1}\in K$} 
 such that,
 for each $i$, $p_i=(a_{i-1},a_i,\ldots,a_{i-1+h})$.
 Thus, $(a_0,a_1,\ldots,a_{h+\ell-1})$ is a partial solution of $F$ of length $\ell$.
 \end{proof}

  \begin{definition}[cf.~\cite{OPS}]
For $\ell\in\mathbb N$ or $+\infty$, a sequence of irreducible $\delta$-subvarieties $(Y_1,\ldots,Y_\ell)$ in $\mathbb{A}^H$ is said to be {\em a train of length $\ell$} in $X$ if
\begin{enumerate}[label=\arabic*)]
\item for all $i$, $1\leqslant i\leqslant \ell$, we have $Y_i\subseteq\sigma^{i-1}(X)$ and 
\item for all $i$, $1\leqslant i<\ell$, we have $\overline{\pi_1(Y_{i+1})}^{\hyperlink{Kol}{\kol}}=\overline{\pi_2(Y_i)}^{\kol}$.
\end{enumerate}
 \end{definition}
 
  \begin{lemma}\label{lem:66}
 For every train $(Y_1,\ldots,Y_\ell)$ in $X$,  there exists a partial solution $p_1,\ldots,p_\ell$ of $(X,\pi_1,\pi_2)$ such that 
 for all $i$, we have $p_i\in Y_i$. In particular, if there is an infinite train in $X$, then there is a solution of the triple  $(X,\pi_1,\pi_2)$.
 \end{lemma} 
\begin{proof}
The proof is similar to that of \cite[Lemma~6.8]{OPS}. 
To make the paper self-contained, we will give the details below.
To prove the existence of a partial solution of $(X,\pi_1,\pi_2)$ with the desired property,
it suffices to prove the following: 

\begin{claim} \em There exists a nonempty open (in the sense of the Kolchin topology) subset $U\subseteq Y_\ell$  such that for each $p_\ell\in U$, $p_\ell$ 
can
be extended to a partial solution $p_1,\ldots,p_\ell$ of  
 $(X,\pi_1,\pi_2)$ with $p_i\in Y_i\,(\forall i)$.
 \end{claim}
 
 We will prove the Claim by induction on $\ell$.
 For $\ell=1$, take $U=Y_1$. 
 Since each point in $Y_1$ is a partial solution of $(X,\pi_1,\pi_2)$ of length 1, the Claim holds for $\ell=1$.
 Now suppose we have proved the Claim for $\ell-1$.  
 So there exists a nonempty open subset $U_0\subseteq Y_{\ell-1}$    satisfying the desired property.
 Since $Y_{\ell-1}$ is irreducible, $U_0$ is dense in $Y_{\ell-1}.$
 So, $\pi_2(U_0)$ is dense in $\overline{\pi_2(Y_{\ell-1})}^{\hyperlink{Kol}{\kol}}=\overline{\pi_1(Y_{\ell})}^{\kol}$. 
Since $U_0$ is $\delta$-constructible, $\pi_2(U_0)$ is $\delta$-constructible too.
So, $\pi_2(U_0)$ contains a nonempty open subset of $\overline{\pi_1(Y_{\ell})}^{\kol}$.

Since $\pi_1(Y_{\ell})$ is $\delta$-constructible  and dense in $\overline{\pi_1(Y_{\ell})}^{\kol}$, 
$\pi_2(U_0)\cap\pi_1(Y_{\ell})\neq\varnothing$ is $\delta$-constructible  and dense in $\overline{\pi_1(Y_{\ell})}^{\kol}$. 
Let $U_1$ be a nonempty open subset of $\overline{\pi_1(Y_{\ell})}^{\kol}$ contained in $\pi_2(U_0)\cap\pi_1(Y_{\ell})$ and \[U_2=\pi_1^{-1}(U_1)\cap Y_{\ell}.\]
Then $U_2$ is a nonempty open subset of  $Y_{\ell}$. 
We will show that for each $p_\ell\in U_2$, there exists $p_{i}\in Y_i$ for $i=1,\ldots,\ell-1$ such that $p_1,\ldots,p_{\ell}$ is a partial solution of  $(X,\pi_1,\pi_2).$ 
 
Since $\pi_1(p_\ell)\in U_1\subset\pi_2(U_0)$, there exists $p_{\ell-1}\in U_0$ such that $\pi_1(p_\ell)=\pi_2(p_{\ell-1}).$ 
Since  $p_{\ell-1}\in U_0$, by the inductive hypothesis, there exists $p_{i}\in Y_i$ for $i=1,\ldots,\ell-1$ such that $p_1,\ldots,p_{\ell-1}$ is a partial solution of  $(X,\pi_1,\pi_2)$ of length $\ell-1$.
So $p_1,\ldots,p_{\ell}$ is a partial solution of  $(X,\pi_1,\pi_2)$ of length $\ell$.
  \end{proof}

 For two trains $Y=(Y_1,\ldots,Y_\ell)$ and $Y'=(Y'_1,\ldots,Y'_\ell)$, denote $Y\subseteq Y'$ if $Y_i\subseteq Y'_i$ for each $i$.
 Given an increasing chain of trains $Y_{i}=(Y_{i,1},\ldots,Y_{i,\ell})$,
 \[\big(\overline{\cup_iY_{i,1}}^{\kol},\ldots,\overline{\cup_iY_{n,i}}^{\kol}\big)\] is a train in $X$ which is an upper bound for this chain. 
 (For each $j$, $\overline{\cup_iY_{i,j}}^{\kol}$ is an irreducible $\delta$-variety in $\sigma^{j-1}$(X).)
 So by Zorn's lemma, maximal trains of length $\ell$ always exist in $X$.
 
Fix an $\ell\in\mathbb N$.
Consider the product 
\[
\textbf{X}_\ell:=X\times\sigma(X)\times\cdots\times \sigma^{\ell-1}(X),
\] 
and denote the projection of $\textbf{X}_\ell$ onto $\sigma^{i-1}(X)$ by $\varphi_{\ell,i}$. Note that 
\[
\mathbf{X}_\ell =\hyperlink{diffspec}{\diffspec} \big(R_X\otimes_{K} R_{\sigma(X)}\otimes_{K}\ldots\otimes_{K} R_{\sigma^{\ell-1}(X)}\big).\]
 Let 
 \begin{equation}
     \label{eq:W}
\mathbf{W}_\ell(X,\pi_1,\pi_2):=\{p\in \textbf{X}_\ell: \pi_2(\varphi_{\ell,i}(p))=\pi_1(\varphi_{\ell,i+1}(p)), i=1,\ldots,\ell-1\}.
 \end{equation}
 Note that
 \[
 \mathbf{W}_\ell = \diffspec\big(R_X\otimes_{R_{\overline{\pi_2(X)}}} R_{\sigma(X)}\otimes_{R_{\overline{\pi_2(\sigma(X))}}}\ldots\otimes_{R_{\overline{\pi_2\left(\sigma^{\ell-2}(X)\right)}}} R_{\sigma^{\ell-1}(X)}\big),\]
 under the injective $(K,\delta)$-algebra homomorphisms, for all $i$, $1\leqslant i \leqslant \ell-1$, \[R_{\overline{\pi_2\left(\sigma^{i-1}(X)\right)}}\to R_{\sigma^{i-1}(X)}\quad\text{and}\quad R_{\overline{\pi_2\left(\sigma^{i-1}(X)\right)}}\to R_{\sigma^{i}(X)}\]
 induced by $\pi_2$ and $\pi_1$, respectively.

 \begin{lemma} \label{lm-traincorrespondence}
 For every irreducible $\delta$-subvariety $W\subset \mathbf{W}_\ell$, \[\big(\overline{\varphi_{\ell,1}(W)}^{\hyperlink{Kol}{\kol}},\ldots,\overline{\varphi_{\ell,\ell}(W)}^{\kol}\big)\] is a train in $X$ of length $\ell$.
 Conversely, for each train $(Y_1,\ldots,Y_\ell)$ in $X$, 
 there exists an irreducible $\delta$-subvariety $W\subseteq \mathbf{W}_\ell$ such that $Y_i=\overline{\varphi_{\ell,i}(W)}^{\kol}$ for each $i=1,\ldots,\ell$.
 \end{lemma}
\begin{proof}The first assertion is straightforward. 
  We will prove the second assertion  by induction on $\ell$.
  For $\ell = 1$, $\mathbf{W}_1= X$, and we can set $W = Y_1$.
  
  Let $\ell > 1$. Apply the inductive hypothesis to the train $(Y_1, \ldots, Y_{\ell - 1})$  
  and obtain an irreducible subvariety $Y' \subset \mathbf{W}_{\ell-1} \subset \mathbf{X}_{\ell - 1}$.
  Then there is a natural embedding of $Y' \times Y_\ell$ into $\mathbf{X}_\ell$.
  Denote $(Y' \times Y_\ell) \cap \mathbf{W}_\ell$ by $\widetilde Y$.
  Since $Y'\subset \mathbf{W}_{\ell-1}$,
  \begin{equation*}\label{eq:reprW}
 \widetilde Y = \{ p \in Y' \times Y_\ell \:|\: \pi_2\left( \varphi_{\ell, \ell - 1}(p) \right) = \pi_1\left( \varphi_{\ell, \ell}(p) \right) \}.
 \end{equation*}
Let 
  \begin{equation}\label{eq:reprZ}
  Z := \overline{\pi_2\left( \varphi_{\ell - 1, \ell - 1}(Y') \right)} = \overline{\pi_1(Y_\ell)}.
  \end{equation}
  Then we have a $(k,\delta)$-isomorphism \[R_{Y'}\otimes_{R_Z}R_{Y_\ell} \to R_{\widetilde{Y}}\]
  under the $(k,\delta)$-algebra homomorphisms   $i_1:R_Z \to R_{Y'}$ and $i_2:R_Z \to R_{Y_\ell}$ induced by  $\pi_2\circ \varphi_{\ell - 1, \ell - 1} $ and $\pi_1$, respectively.
  Equality~\eqref{eq:reprZ} implies that $i_1$ and $i_2$ are injective.
 Denote the fields of fractions of $R_{Y'}$, $R_{Y_\ell}$, and $R_Z$ by $E$, $F$, and $L$, respectively.
  Let $\mathfrak{p}$ be any prime differential ideal in $E \otimes_L F$, \[R := (E \otimes_L F) / \mathfrak{p},\]and $\pi\colon E \otimes_L F \to R$ be the canonical homomorphism.
  Consider the natural homomorphism $i \colon R_{Y'}\otimes_{R_Z}R_{Y_\ell} \to E \otimes_L F$.
  Since $1 \in i(R_{Y'}\otimes_{R_Z}R_{Y_\ell})$, the composition $\pi \circ i$ is a nonzero homomorphism.
Since $i_1$ and $i_2$ are injective, the natural homomorphisms $i_{Y'} \colon R_{Y'} \to R_{Y'}\otimes_{R_Z}R_{Y_\ell}$ and $i_{Y_\ell} \colon R_{Y_\ell} \to R_{Y'}\otimes_{R_Z}R_{Y_\ell}$ are injective as well.
  We will show that the compositions \[\pi\circ i \circ i_{Y'} \colon R_{Y'} \to R\quad \text{and}\quad 
  \pi\circ i \circ i_{Y_\ell} \colon R_{Y_\ell} \to R\] are injective.
  Introducing the natural embeddings $i_E \colon E \to E\otimes_L F$ and $j_{Y'} \colon R_{Y'} \to E$, we can rewrite
  \[
  \pi\circ i \circ i_{Y'} = \pi \circ i_E \circ j_{Y'}.
  \]
  The homomorphisms $i_E$ and $j_{Y'}$ are injective.
  The restriction of $\pi$ to $i_E(E)$ is also injective since $E$ is a field.
  Hence, the whole composition $\pi \circ i_E \circ j_{Y'}$ is injective. 
  The argument for $\pi\circ i \circ i_{Y_\ell}$ is analogous. Let 
  \[S :=\big(R_{Y'}\otimes_{R_Z}R_{Y_\ell}\big)\big/ \bigl(\mathfrak{p}\cap \big(R_{Y'}\otimes_{R_Z}R_{Y_\ell}\big)\bigr),\]
  which is a domain, and the  homomorphisms $\pi\circ i \circ i_{Y'}: R_{Y'}\to S$ and $\pi\circ i \circ i_{Y_\ell} : R_{Y_\ell}\to S$ are injective. We let \[W:=\hyperlink{diffspec}{\diffspec}{S}.\]
  For every $i$, $1 \leqslant i < \ell$, the homomorphism \[\varphi_{\ell, i}^\sharp = (\pi\circ i \circ i_{Y'})\circ\varphi_{\ell - 1, i}^\sharp : R_{Y_i}\to R_{Y'}\to S \] is injective as a composition of two injective homomorphisms. Hence,
  the restriction $\varphi_{\ell, i} \colon W \to Y_i$ is dominant.
 \end{proof}

 \subsection{
Technical bounds}\label{subsec:bounds}
 \subsubsection{Number of prime components in differential varieties}\label{sec:numberofcomponents}
 In this section, we fix a $\delta$-field $k$ 
and $\bm{x}= x_1, \ldots, x_n$. For a commutative ring $R$ and subsets $I$ and $S$ of $R$, we let $I:S = \{r\in R\mid \exists\, s \in S:\, rs\in I\}$. 

\begin{lemma}\label{lem:separants}
  There exists a computable function $G(n, r, D)$ such that,
  for every $r \in \mathbb{Z}_{\geqslant 0}$ and a prime ideal $I \subset 
 k[\bm{x}_{r,0}]
  $ such that $I = {\sqrt{\langle I\rangle^{(\hyperlink{diffideal}{\infty})}} \cap k[\hyperlink{Rrs}{\bm{x}_{r,0}}]}$ and $\deg I \leqslant D$, there exists 
  $f \in k[\bm{x}_{r,0}]\setminus I$ such that
  \begin{itemize}
    \item $\langle I\rangle^{(\infty)} : f^{\infty}$ is a prime differential ideal;
    \item $\deg f \leqslant G(n, r, D)$.
  \end{itemize}
\end{lemma}

\begin{proof}
  Compute a regular decomposition of $\sqrt{\langle I\rangle^{(\infty)}} $ using the Rosenfeld-Gr\"obner algorithm with an orderly ranking:
  \[
 \sqrt{\langle I\rangle^{(\infty)}}  = \big(\langle C_1\rangle^{(\infty)} : H_1^\infty\big) \cap \ldots \cap \big(\langle C_N\rangle^{(\infty)} : H_N^\infty\big) .
  \]
  Since the Rosenfeld-Gr\"obner algorithm with an orderly ranking does not increase the orders of the polynomials, $C_1, \ldots, C_N \subset k[\bm{x}_{r,0}]$.
  Since $I$ is prime, there exists $i$, say $i = 1$, such that 
  \begin{equation}\label{eq:proj}
 k[\bm{x}_{r,0}]\cap \big(\langle C_1\rangle^{(\infty)}:H_1^\infty\big)  = I.
  \end{equation}
    We show that $J :=\langle C_1\rangle^{(\infty)} : H_1^\infty $ is a prime differential ideal.
  Suppose $P_1, P_2\in k[\bm{x}_{\infty,0}]$ with $P_1P_2 \in J$.
  Let $\overline{P_i}\,(i=1,2)$ be the partial remainder of $P_i$ with respect to $C_1$~\citep[p.~396]{Rozenfeld}. 
  Then $\overline{P_1}\cdot\overline{P_2}\in J$.
  Due to Rosenfeld's lemma~\citep[p.~397]{Rozenfeld}, 
  \[
  \overline{P_1}\cdot\overline{P_2} \in k[\bm{x}_{\infty,0}]  \cdot ((C_1) : H_1^\infty) \subseteq k[\bm{x}_{\infty,0}] \cdot (I : H_1^{\infty}) = k[\bm{x}_{\infty,0}]\cdot I.
  \]
  Since $k[\bm{x}_{\infty,0}]\cdot I$ is prime, at least one of $\overline{P_i}$ belongs to $k[\bm{x}_{\infty,0}]\cdot I \subset J$.
  So $P_1\in J$ or $P_2\in J$.
  Thus, $J$ is prime.
    Equality~\eqref{eq:proj} together with $C_1 \subset k[\bm{x}_{r,0}]$ imply that
  \[
    \langle C_1\rangle^{(\infty)} : H_1^\infty  = \langle I\rangle^{(\infty)} : H_1^\infty,
  \]
  so $\langle I\rangle^{(\infty)} : H_1^\infty$ is a prime differential ideal.
  Since the differential polynomials from $C_1$ together with some of their derivatives constitute a triangular set for  $I$, \cite[Theorem~1]{Schost2003} implies that the degree of every initial and every separant of $C_1$ is bounded by
  \[
    (2(n(r+1))^2 + 2)^{n(r+1)} D^{2n(r+1) + 1} + D.
  \]
  Since there are at most $nr$ elements in $C_1$, setting \[G(n, r, D) := 2n(r+1) (2(n(r+1))^2 + 2)^{n(r+1)} D^{2n(r+1) + 1} + 2n(r+1)D\quad \text{and}\quad f := H_1\] finishes the proof of the lemma.
\end{proof}

%%%%%%%%%%%%%%%%%%

\begin{lemma}\label{lem:intersection}
  For every differential ring $R$, subring $S \subset R$, and ideals $I, P_1,\ldots,P_\ell \subset S$ such that $I = P_1 \cap \ldots \cap P_\ell$, we have
  \[
   \sqrt{\langle I\rangle^{\hyperlink{diffideal}{(\infty)}}}  = \sqrt{\langle P_1\rangle^{(\infty)}}  \cap \ldots \cap \sqrt{\langle P_\ell\rangle^{(\infty)}} 
  \]
\end{lemma}

\begin{proof}
  \cite[Lemma~8]{OPV2018} implies that, for every $s > 0$,
  \[
    P_1^{(s)} \cap \ldots \cap P_\ell^{(s)} \subset \sqrt{(P_1 \cap \ldots \cap P_\ell)^{(\ell s)}}.
  \]
  Since taking the radical commutes with intersections, we have
  \[
    \sqrt{P_1^{(s)}} \cap \ldots \cap \sqrt{P_\ell^{(s)}} \subset \sqrt{(P_1 \cap \ldots \cap P_\ell)^{(\ell s)}}.
  \]
  We also have
  \[
    \sqrt{(P_1 \cap \ldots \cap P_\ell)^{(\ell s)}} \subset \sqrt{P_1^{(\ell s)} \cap \ldots\cap P_\ell^{(\ell s)}} = \sqrt{P_1^{(\ell s)}} \cap \ldots\cap \sqrt{P_\ell^{(\ell s)}}.
  \]
  Taking $s = \infty$, we obtain
  \[
  \sqrt{\langle P_1\rangle^{(\infty)}}  \cap \ldots \cap \sqrt{\langle P_\ell\rangle^{(\infty)}}  \subset \sqrt{\langle I\rangle^{(\infty)}} \subset  \sqrt{\langle P_1\rangle^{(\infty)}}  \cap \ldots \cap \sqrt{\langle P_\ell\rangle^{(\infty)}} .\qedhere
  \]
\end{proof}

%%%%%%%%%%%

\begin{lemma}\label{lem:projection}
  There exists a  computable function $F(n, r, m, D)$ such that, for every $r, m, D$ and  radical ideal $J \subset k[\hyperlink{Rrs}{\bm{x}_{r,0}}]$ of dimension $m$ and degree $D$,
  \[
    \deg \left(k[\bm{x}_{r,0}] \cap \sqrt{\langle J\rangle^{(\infty)}}\right) \leqslant F(n, r, m, D).
  \]
\end{lemma}

\begin{proof}
  \cite[Theorem~3]{OPV2018} together with \cite[Proposition~3]{Heintz} imply that 
  \[
   k[\bm{x}_{r,0}] \cap \sqrt{\langle J\rangle^{(\infty)}} = k[\bm{x}_{r,0}] \cap  \sqrt{J^{(B)}},
  \]
  where $B := D^{n(r+1) 2^{m + 1}}$.
  Thus,  $\deg \left(k[\bm{x}_{r,0}] \cap \sqrt{\langle J\rangle^{(\infty)}}\right)\leqslant\deg  \sqrt{J^{(B)}}$.
  The B\'ezout inequality implies that $\deg  \sqrt{J^{(B)}}\leqslant D^{n(r+1)B}$.
  Thus, we can finish the proof of the lemma by setting $F(n, r, m, D) = D^{n(r+1)B}$.
\end{proof}

%%%%%%%%%%%%%%%%%%

\begin{prop}\label{prop:main}
  There is a computable function $C(n, r, m, D)$ such that, for every nonnegative integers $r, m, D$ and every  radical  ideal $I \subset k[\bm{x}_{r,0}]$ such that 
  \begin{itemize}   
   \item $\deg I \leqslant D$,
   \item $\dim I \leqslant m$,
   \item $I = \sqrt{\langle I\rangle^{\hyperlink{diffideal}{(\infty)}}} \cap k[\bm{x}_{r,0}]$,
   \end{itemize}
   the number of prime components of $\sqrt{\langle I\rangle^{(\infty)}}$ does not exceed $C(n, r, m, D)$.
\end{prop}

\begin{proof}
  We fix  
  $r$ 
  for the proof and will prove the proposition by constructing the function $C(r, m, D)$ by induction on a tuple $(m, D)$ with respect to the lexicographic ordering.
  
  Consider the base case $m = 0$.
  Then there are at most $D$ possible values for $(x_1, \ldots, x_n)$ and every prime component of $\sqrt{\langle I\rangle^{\hyperlink{diffideal}{(\infty)}}}$ is the maximal differential ideal corresponding to one of these values.
  Thus, the proposition is true for $C(n, r, 0, D) = D$.
  
  Consider $m > 0$. 
  \underline{If $I$ is not prime}, then Lemma~\ref{lem:intersection} implies that the number of prime components of $\sqrt{\langle I\rangle^{(\infty)}}$ does not exceed
  \begin{equation}\label{eq:bound_reducible}
    \max\limits_{\ell}\max\limits_{D_1 + \ldots + D_\ell = D} \left( C(n, r, m, D_1) + \ldots + C(n, r, m, D_\ell) \right),
  \end{equation}
  where all $C(n, r, m, D_i)$ are already defined by the inductive hypothesis.
  
  \underline{Consider the case of prime $I$.} 
  Lemma~\ref{lem:separants} implies that there exists $f \in k[\bm{x}_{r,0}]\setminus I$ such that $\deg f \leqslant G(n, r, D)$ and $\langle I\rangle^{(\infty)} : f^{\infty}$ is a prime differential ideal.
  Every prime component of $\sqrt{\langle I\rangle^{(\infty)}}$ either is equal to $\langle I\rangle^{(\infty)} : f^{\infty}$ or contains $f$.
  In the latter case, the component is a component of $\sqrt{\langle I, f\rangle^{(\infty)}}$.
  Let $J :=\sqrt{\langle I,f\rangle^{(\infty)}} \cap k[\bm{x}_{r,0}]$.
  Then $\dim J \leqslant m - 1$ and Lemma~\ref{lem:projection} implies that $\deg J \leqslant F(r, m - 1, G(n, r, D))$.
  Then the number of prime components of $\sqrt{\langle I\rangle^{(\infty)}}$ does not exceed
  \begin{equation}\label{eq:bound_irreducible}
    1 + C(n, r, m - 1, F(r, m - 1, G(n, r, D))).
  \end{equation}
  Thus, one can define $C(n, r, m, D)$ to be the maximum of~\eqref{eq:bound_reducible} and~\eqref{eq:bound_irreducible}.
\end{proof}

%%%%%%

Proposition~\ref{prop:main} and Lemma~\ref{lem:projection} imply the following corollary.

\begin{cor}\label{cor:numberofcomponents}
  For every  radical ideal $I \subset k[\hyperlink{Rrs}{\bm{x}_{r,0}}]$ 
  of dimension at most $m$ and degree at most $D$, the number of prime components of $\sqrt{\langle I\rangle^{(\infty)}}$ does not exceed
  \[
    C(n, r, m, F(n, r, m, D)).
  \]
\end{cor}

 \subsubsection{Bound for trains}\label{sec:bound_trains}
 Now we try to give a bound so that the existence of a maximal train of certain length in $X$ will definitely guarantee the existence of at least one infinite train in $X$.
 \begin{definition}[Kolchin polynomials for $\delta$-varieties and trains]
 \begin{itemize}
 \item[]
 \item  The {\em Kolchin polynomial} of an irreducible $\delta$-variety $V=\mathbb{V}(F)$, where $F \subset K[\hyperlink{Rrs}{\bm{y}_{\infty,0}}]$, where $\bm{y}=y_1,\ldots,y_n$, 
 is the unique numerical polynomial $\omega_V(t)$  such that there exists $t_0\geqslant 0$ such that, for all $t \geqslant t_0$ and
 the
 generic point  
$\bm{a}$
 of $V$  (see Definition \ref{def:varietyanddiffspec}),  $\omega_V(t)=\trdeg K(\bm{a}_{t,0})/K$.
     \item
  The Kolchin polynomial of a $\delta$-variety is defined to be the maximal Kolchin polynomial of its irreducible components.
  \item
  An irreducible component $X_1$ of a $\delta$-variety $X$ is called a {\em generic component} if $\omega_{X_1}(t)=\omega_{X}(t)$.
  \item
  We define the {\em Kolchin polynomial of a train} $Y=(Y_1,\ldots,Y_\ell)$ in $X$ as \[\omega_Y(t):=\min_i\omega_{Y_i}(t).\]
  \end{itemize}
  \end{definition}

 \begin{remark} The Kolchin polynomial of an irreducible $\delta$-variety $V$ is of the form (see~\cite[formula~(2.2.6)]{KLMP} and \cite[Theorem~II.12.6(d)]{Kol})
\[\omega_V(t)=\hyperlink{ddim}{\ddim}(V)\cdot(t+1)+\hyperlink{ord}{\ord}(V).\]
\end{remark}
The following lemma shows how the coefficients of a Kolchin polynomial change under a projection.

\begin{lemma} \label{lm-projection}
Let $V\subset\mathbb A^n$ be an irreducible $\delta$-variety and $\pi_1: \mathbb A^n\longrightarrow \mathbb A^{n-1}$ be the projection to the first $n-1$ coordinates.
Then we have \[\hyperlink{ddim}{\ddim}\big(\overline{\pi_1(V)}^{\hyperlink{Kol}{\kol}}\big)\leqslant {\ddim}(V)\ \ \text{and}\ \  \hyperlink{ord}{\ord}\big(\overline{\pi_1(V)}^{\kol}\big)\leqslant \ord(V).\]
\end{lemma}
\proof
Let 
$\bm{a}$
be a generic point of $V$.
Then 
$\pi_1(\bm{a})$
is a generic point of $W:=\overline{\pi_1(V)}^{\kol}$.
Clearly,  \[\omega_W(t)\leqslant \omega_V(t)\quad  \text{and}\quad {\ddim}(W)\leqslant {\ddim}(V).\]
So, we have \[\ddim(W)=\ddim(V)\implies \ord(W)\leqslant \ord(V).\]
It, therefore, suffices to show that\[\ddim(W)<\ddim(V) \implies \ord(W)\leqslant \ord(V).\]
Suppose  $\ddim(W)<\ddim(V)=d$.
Then we have \[\ddim(W)=\ddim(V)-1=d-1.\]
Since the order of an irreducible  $\delta$-variety $V$ is equal to the maximal relative order of $V$ with respect to a parametric set,
without loss of generality, suppose \[
\ord(W)=\trdeg\, K\big(\pi_1(\bm{a})_{\infty,0}\big)\big/ K\big(\pi_{n-(d-1)}(\bm{a})_{\infty,0}\big),
\]
where $\pi_{n-(d-1)} : \mathbb{A}^n\to \mathbb{A}^{d-1}$ is the projection to the first $d-1$ coordinates.  
Since \[
\hyperlink{deltrdeg}{\deltrdeg} K\big(\bm{a}_{\infty,0}\big)\big/ K = 1 + \deltrdeg  K\big(\pi_1(\bm{a})_{\infty,0}\big)\big/ K,
\]
 $a_n$ is $\delta$-transcendental over 
 $K(\pi_1(\bm{a})_{\infty,0})$, i.e., $\deltrdeg  K\big(\bm{a}_{\infty,0}\big)\big/ K\big(\pi_1(\bm{a})_{\infty,0}\big)=1$.
Therefore, we have
\begin{align*}
\ord(W)&= \trdeg K\big(\pi_1(\bm{a})_{\infty,0}\big)\big/ K\big(\pi_{n-(d-1)}(\bm{a})_{\infty,0}\big)  \\
&= \trdeg K\big({(a_n)}_{\infty,0}\big) \big(\pi_1(\bm{a})_{\infty,0}\big)\big/ K\big({(a_n)}_{\infty,0}\big)\big(\pi_{n-(d-1)}(\bm{a})_{\infty,0}\big)   = \hyperlink{relord}{\ord_{y_1,\ldots,y_{d-1},y_n}}(V) \leqslant \ord(V). \qedhere
\end{align*}

\begin{lemma} \label{lm-rittnumber}
For all $s \in \mathbb{Z}_{\geqslant 0}$ and $F \subset k[\bm{y}_{s,0}]$, where $\bm{y}=y_1,\ldots,y_n$,
the order of each component of $\mathbb V(F)$ is bounded by  $ns$.
\end{lemma}
\begin{proof}
It follows directly by \cite[p. 135]{Ritt} and \cite[Theorem 2.11]{GLY}.
\end{proof}

 \begin{definition} 
 For all
 \begin{itemize}
     \item
non-negative integers $n,s,h,d$,
\item $\mathbb{Z}_{\geqslant 0}$-valued polynomials $\omega \in \mathbb{Z}[t]$,
\end{itemize}
 we define $B(\omega,n,s,h,d)$ to be the smallest $M\in\mathbb N\cup\{\infty\}$ such that, for every triple $(X,\pi_1,\pi_2)$ with
\[
X=\mathbb V(F)\subseteq\mathbb A^{n(h+1)},\ \ F \subset k[\hyperlink{Rrs}{\bm{y}_{s,h}}],\ \bm{y}=y_1,\ldots,y_n,\ \  
\deg(F)\leqslant d,
\] 
if there exists a 
train in $X$ of length $M$ and
Kolchin polynomial at least $\omega$, then there exists  an  infinite train in $X$.
 \end{definition}

\begin{remark} For all
non-negative integers $n,s,h,d$ and
$\mathbb{Z}_{\geqslant 0}$-valued polynomials $\omega \in \mathbb{Z}[t]$,
\[B(\omega(t),n,s,h,d)\leqslant B(0,n,s,h,d).
\]
\end{remark}

 In the following, we will show that $B(\omega(t),n,s,h,d)$ is finite   for all $\omega(t)\geqslant 0$ and the numerical data $n,s,h,d$  and is also  bounded by a computable function in these numerical data.
 For the ease of notation, we denote \[L(n,r,d):=C\big(n,r,n(r+1),F(n,r,n(r+1),d)\big),\]
 which is computable.
 So by Corollary \ref{cor:numberofcomponents},
given a system $S$ of $\delta$-polynomials in $n$ $\delta$-variables of order bounded by $r$ and degree bounded by $d$, the number of components of the $\delta$-variety $\mathbb V(S)$ is bounded by $L(n,r,d)$.
By Lemma \ref{lm-traincorrespondence}, the number of maximal trains in $X$ of length $\ell$ is bounded by $L(n(h+\ell),s,d)$.
We now define two  increasing  
sequences $(A_i(n,h,s,d))_{i\in\mathbb N}$ and $(\tau_i(n,h,s,d))_{i\in\mathbb N}$ as follows:\hypertarget{tau}{}
\begin{gather}
\begin{split}
\label{eq:defA}
A_0=L(n(h+1),s,d)+1,\,\, A_{i + 1} = A_i + L(n(h+1)A_i,s,d)\,\,(\text{for } i\geqslant 0)\\
 \tau_0 = ns(h+1),\,\, \tau_{i+1}  = \tau_i + ns(h+1)A_{\tau_i}+1\,\,(\text{for } i\geqslant 0).
 \end{split}
 \end{gather}

\begin{lemma} \label{thm:bound}
  We have
  \[B(0,n,s,h,d)\leqslant A_{\tau_{n(h+1)}
  (n,h,s,d)}(n,h,s,d),\] 
which is computable.
 \end{lemma}
 \begin{proof} Temporarily, fix $X$.
  By Corollary~\ref{cor:numberofcomponents}, 
we  know upper bounds for the number of irreducible components of X and for the number of maximal trains in $X$ of any fixed length.
 The main idea of the proof is to construct a decreasing chain of Kolchin polynomials $\omega_0(t)>\omega_1(t)>\cdots$ and for each $\omega_i(t)$,
 give an upper bound $B_i$ for  $B(\omega_i(t),n,s,h,d)$. Since the Kolchin polynomials are well-ordered, the decreasing chain will stop at some $\omega_J(t)=0$.

Let $\omega_0(t)=\omega_X(t)$.
Let $B_0$ be the number of generic components of the $\delta$-variety $X$ plus 1.
Consider a train $(Y_1,\ldots,Y_{B_0})$ in $X$ of Kolchin polynomial at least $\omega_0(t)$.
So for each $i$,  $\sigma^{-i+1}(Y_i)$ is a $\delta$-subvariety of $X$ with Kolchin polynomial at least $\omega_X(t)$,
so $\sigma^{-i+1}(Y_i)$ must be a generic component of $X$.
Since $X$ has only $B_0-1$ generic components, there exists $a<b\in \mathbb N$ such that $\sigma^{-a+1}(Y_a)=\sigma^{-b+1}(Y_b)$,
which implies $Y_b=\sigma^{b-a}(Y_a)$.
Thus, we can construct an infinite train $$(\ldots,Y_a,Y_{a+1},\ldots,Y_{b-1},\sigma^{b-a}(Y_a),\sigma^{b-a}(Y_{a+1}),\ldots,\sigma^{b-a}(Y_{b-1}),\ldots).$$

 Suppose $\omega_i(t)$ and $B_i$ have been constructed.
 We now try to  do it for $i+1$.
 Let
 \begin{equation}\label{eq:Biplus1}B_{i+1}=B_i+D_i,
 \end{equation}where $D_i$ is the number of maximal trains in  $X$ of length $B_i$.
 Consider the  fibered product $\mathbf{W}_{B_i}(X, \pi_1, \pi_2)$, as in~\eqref{eq:W}, 
 and, for each irreducible component $W$ of $\mathbf{W}_{B_i}$,
 denote \[Y_W=\big(\overline{\varphi_{B_i,1}(W)}^{\kol},\ldots,\overline{\varphi_{B_i,B_i}(W)}^{\kol}\big)\] to be the train corresponding to $W$.
 Let \[\omega_{i+1}(t):=\max\big\{\omega_{Y_W}(t)\:\big|\:  \omega_{Y_W}(t)<\omega_i(t),  W \text{  is a component of } 
 \mathbf{W}_{B_i}
 \big\},\]
 and set $\max\varnothing=0.$

 Consider a maximal train $(Y_1,\ldots,Y_{B_{i+1}})$ in $X$ with Kolchin polynomial at least $\omega_{i+1}(t)$.
 We will show this $B_{i+1}$ works.
 Introduce $D_i+1$ trains $Z^{(1)},\ldots,Z^{(D_i+1)}$ of length $B_i$ in $X,\sigma(X),\ldots,\sigma^{D_i}(X)$, respectively, such that for each $j$,
 $$Z^{(j)}=\big(Z^{(j)}_1,\ldots,Z^{(j)}_{B_i}\big):=(Y_j,\ldots,Y_{j+B_i-1}).$$
 Then for each $j$, consider a maximal train $\tilde{Z}^{(j)}$ of length $B_i$ containing $Z^{(j)}$.
 So $\sigma^{-j+1}(\tilde{Z}^{(j)})$ is a maximal train of length $B_i$ in $X$.
 There are two cases to consider:
 \begin{equation}\tag{Case 1}
 \big\{\omega_{Y_W}(t)\:\big|\:  \omega_{Y_W}(t)<\omega_i(t),  W \text{  is a component of } 
 \mathbf{W}_{B_i}\big\}=\varnothing.\end{equation}
 In this case, $\omega_{i+1}(t)=0$,
 and for each $j$, \[\omega_{\sigma^{-j+1}(\tilde{Z}^{(j)})}(t)\geqslant\omega_i(t).\]
 By the construction of $B_i$, we could construct an infinite train through each $\sigma^{-j+1}(\tilde{Z}^{(j)})$.
\begin{equation}\tag{Case 2}\big\{\omega_{Y_W}(t)\:\big|\:  \omega_{Y_W}(t)<\omega_i(t),  W \text{  is a component of }
\mathbf{W}_{B_i}
\big\}\neq\varnothing.
\end{equation}
 If there exists some $j_0$ such that $\omega_{\sigma^{-j_0+1}(\tilde{Z}^{(j_0)})}(t)\geqslant\omega_i(t)$, then
by the construction of $B_i$, we could construct an infinite train through this $\sigma^{-j_0+1}(\tilde{Z}^{(j_0)})$.
 Suppose now that,  for each $j$, \[\omega_{\sigma^{-j+1}(\tilde{Z}^{(j)})}(t)=\omega_{i+1}(t).\]
 Since there are only $D_i$ number of maximal trains in  $X$ of length $B_i$,
  there exist $a<b$ such that \[\sigma^{-a+1}(\tilde{Z}^{(a)})=\sigma^{-b+1}(\tilde{Z}^{(b)}).\]
 Since $\omega_{\sigma^{-a+1}(\tilde{Z}^{(a)})}(t)=\omega_{i+1}(t)$, there exists $l$ such that \[\omega_{\sigma^{-a+1}(\tilde{Z}^{(a)}_l)}(t)=\omega_{i+1}(t).\]
 Then \[\omega_{\sigma^{-a+1}({Z}^{(a)}_l)}(t)=\omega_{i+1}(t),\] for $\sigma^{-a+1}({Z}^{(a)}_l)\subseteq\sigma^{-a+1}(\tilde{Z}^{(a)}_l)$ and
 the  Kolchin polynomial of $(Y_1,\ldots,Y_{B_{i+1}})$ is at least $\omega_{i+1}(t)$.
 So we have \[\sigma^{-a+1}({Z}^{(a)}_l)=\sigma^{-a+1}({\tilde{Z}}^{(a)}_l).\]
 Similarly, we can show \[\sigma^{-b+1}({Z}^{(b)}_l)=\sigma^{-b+1}({\tilde{Z}}^{(b)}_l).\]
 So \[\sigma^{-a+1}(Y_{a+l-1})=\sigma^{-a+1}({Z}^{(a)}_l)=\sigma^{-a+1}({\tilde{Z}}^{(a)}_l)=\sigma^{-b+1}({\tilde{Z}}^{(b)}_l)=\sigma^{-b+1}({Z}^{(b)}_l)=\sigma^{-b+1}(Y_{b+l-1}).\]
Thus, we have \[Y_{b+l-1}=\sigma^{b-a}(Y_{a+l-1}).\]
Therefore, we can construct an infinite sequence
$$\Big(Y_1,\ldots,Y_{a+l-1},,\ldots,Y_{b+l-2},\sigma^{b-a}(Y_{a+l-1}), \dots,\sigma^{b-a}(Y_{b+l-2}),\ldots\Big).$$

 As we described in the first paragraph, as the process goes on, we has constructed a decreasing chain of Kolchin polynomials
 \[\omega_0(t)=\omega_X(t)>\omega_1(t)>\omega_2(t)>\cdots.\]
 Since the Kolchin polynomials are well-ordered, this chain is finite,
 so the above process will stop at step $J$ at which we could get $\omega_J(t)=0$,
 either in the case in which \[\big\{\omega_{Y_W}(t)\:\big|\:  \omega_{Y_W}(t)<\omega_{J-1}(t),  W \text{  is a component of } 
 \mathbf{W}_{B_{J-1}}
 \big\}=\varnothing,\]
 or in the case in which the set is nonempty and the maximal Kolchin polynomial in the set is 0.
 
  By Lemma~\ref{lm-traincorrespondence} and Corollary \ref{cor:numberofcomponents},  
for the number $D_i$ of maximal trains in $X$ of length $B_i$, we have
\begin{equation}\label{eq:bound_D}
D_i \leqslant L(n(h+1)B_i, s, d),\ \text{ so }\ B_{i + 1} \leqslant B_i + L(n(h+1)B_i, s, d).
\end{equation}
By Corollary~\ref{cor:numberofcomponents}, we have 
\[
  B_0 \leqslant L(n(h + 1), s, d) + 1.
\]
  For each $i$, $0 \leqslant i \leqslant J$, let $a_i$ and $b_i$ be such that 
  \[
    \omega_i(t) = a_i(t + 1) + b_i.
  \]
  For $i = 0$, we have $a_0 = \hyperlink{ddim}{\ddim}(X)$ and $b_0 = \hyperlink{ord}{\ord}(X)$.
  For every $j$, $0 \leqslant j \leqslant a_0$, we define $i_j$ to be the largest integer in $[0,J]$ such that $a_0 - a_{i_j} \leqslant j$.
  Then $J = i_{a_0}$.
  The decreasing of the Kolchin polynomials implies that, for all $j$,  $0 \leqslant j <  a_0$:
  \begin{itemize}
      \item we have
  \begin{equation}\label{eq:bound_i}
      i_0 \leqslant b_0 \quad\text{ and}\quad  i_{j + 1} \leqslant i_j + b_{i_j} + 1.
  \end{equation}
  \item by the definition of $\omega_{i_j + 1}(t)$ and Lemma \ref{lm-projection}, $b_{i_j+1}$ is bounded by the maximal order of the components of $W_{B_{i_j}}$, so 
  \item by  
  Lemma~\ref{lm-rittnumber}, 
  \begin{equation}\label{eq:bound_b}
      b_{i_j + 1} \leqslant n s(h+1)B_{i_j}.
  \end{equation}
  \end{itemize}
  Comparing the recursive formulas~\eqref{eq:defA} with inequalities~\eqref{eq:bound_D}, \eqref{eq:bound_i}, and \eqref{eq:bound_b}, we see that 
  \begin{itemize}
    \item $B_i \leqslant A_i$ for every $i$, $0 \leqslant i \leqslant J$;
    \item $i_j \leqslant \hyperlink{tau}{\tau_j}$ for every $j$, $0 \leqslant j \leqslant a_0 = \ddim(X)$.
  \end{itemize}
  Thus,
  \[
  B_J = B_{i_{\ddim(X)}} \leqslant A_{i_{\ddim(X)}} \leqslant A_{\tau_{\ddim(X)}}\leqslant A_{\tau_{n(h+1)}} .\qedhere
  \]
 \end{proof}

As a consequence, we have the following result.

 \begin{cor} \label{cor:620}
For all $s,h \in\mathbb{Z}_{\geqslant 0}$ and $F \subset k[\hyperlink{Rrs}{\bm{y}_{s,h}}]$,
 $F = 0$ has a solution in 
 $K^\mathbb Z$ 
 if and only if $F = 0$ has a partial solution of length $D:=A_{\tau_{n(h+1)}(n,h,s,d)}(n,h,s,d)$.
 \end{cor}
\begin{proof}
 Let $X\subset\mathbb A^H$ be the $\delta$-variety defined by $F=0$ regarded as a system of $\delta$-equations in $\bm{y}, \sigma(\bm{y}),\ldots, \sigma^h(\bm{y})$, where $H=n(h+1)$.
 By Lemmas~\ref{lem:64}  and~\ref{lem:66}, 
 $F = 0$ has a partial solution of length $D$ (resp. $F = 0$ has a solution in $K^\mathbb Z$
 ) if and only if there exists a train of length $D$ in $X$ (resp., there exists an infinite train in $X$).
 By Lemma~\ref{thm:bound}, if there exists a train of length $D$ in $X$, then there  exists a infinite train in $X$. So the assertion holds.
 \end{proof}

 \subsection{Proof of Theorem~\ref{thm:main}}\label{sec:main_proof}
We will prove a more refined version of Theorem~\ref{thm:main}:
\begin{theorem}\label{thm:main3}  For all non-negative integers $r,s,h,d$, there exists a computable $B=B(r,s,h,d)$ such that, for all:  \begin{itemize}
\item non-negative integers $q$,
   \item $\delta$-$\sigma$-fields
    $k$,   
\item sets of $\delta$-$\sigma$-polynomials $F\subset k[\hyperlink{Rrs}{\bm{x}_{s,h},\bm{y}_{s,h}}]$, where $\bm{x} =x_1,\ldots,x_q$,  $\bm{y}=y_1,\ldots,y_r$, and
$\deg_{\bm{y}}F\leqslant d$, 
   \end{itemize}
   we have 
      \[\big\langle \sigma^i(F)\mid i\in \mathbb{Z}_{\geqslant 0} \big\rangle^{\hyperlink{diffideal}{(\infty)}}\cap k[\hyperlink{kinfty}{\bm{x}_\infty}]\ne\{0\} \iff \langle \sigma^i(F)\mid i\in [0,B] \big\rangle^{\hyperlink{diffideal}{(B)}}\cap k[\hyperlink{kB}{\bm{x}_B}]\ne\{0\}.\]
 \end{theorem}
 \begin{proof}
 The ``$\impliedby$'' implication is straightforward. We will prove the ``$\implies$" implication. For this, let 
 $A := A_{\tau_{r(h+1)}(r,h,s,d)}(r,h,s,d)$,
 $B_\delta$ be the bound $B$ from \cite[Theorem~1]{OPV2018} with \[|\bm{\alpha}|\leftarrow r(s+1)(A + h + 1),
 \ \  
 m\leftarrow r(s+1)(A + h + 1),\ \ \text{and}\ \ d \leftarrow d,\]
 and $B := B_\delta + s$.
 By assumption, 
 \begin{equation}\label{eq:assumption}1 \in \big\langle \sigma^i(F)\mid i\in \mathbb{Z}_{\geqslant 0} \big\rangle^{(\infty)}\cdot k(\bm{x}_\infty)[\bm{y}_\infty].
 \end{equation}
  Suppose that 
 \begin{equation}\label{eq:zerointersect}\langle \sigma^i(F)\mid i\in [0, A] \big\rangle^{(B_{\delta})}\cap k[\bm{x}_{B}]=\{0\}.
 \end{equation}
 If  \[1 \in \big\langle \sigma^i(F)\mid i\in [0, A] \big\rangle^{({B_{\delta}})}\cdot k(\bm{x}_{B})[\bm{y}_{\infty,A + h}],\] 
 then there would exist $c_{i, j}\in k(\bm{x}_{B})[\bm{y}_{\infty,A + h}]$ such that 
 \begin{equation}\label{eq:expr1}
  1 = \sum\limits_{i = 0}^{{B_{\delta}} } \sum\limits_{j = 0}^A  \sum\limits_{f\in F} c_{i, j}\delta^i\sigma^j(f).
 \end{equation}
 Multiplying equation~\eqref{eq:expr1} by the common denominator in the variables $\bm{x}_B$, we obtain a contradiction with~\eqref{eq:zerointersect}.  Hence,
 by \cite[Theorem~1]{OPV2018},
 \[1 \notin \big\langle \sigma^i(F)\mid i\in [0, A] \big\rangle^{(\infty)}\cdot k(\bm{x}_{B})[\bm{y}_{\infty, A + h}].\] 
 By the differential Nullstellensatz, there exists a differential field extension $K\supset k(\bm{x}_{\infty})\supset k(\bm{x}_{B})$ such that the system of differential equation 
 \[
 \{\sigma^i(F)=0\mid i\in [0, A]\}
 \] 
 in the unknowns $\bm{y}_{\infty, A + h}$ has a 
 solution in $K$. 
Then the system $F = 0$ has a partial solution of length $A + 1$ in $K$.  However, by Proposition~\ref{prop:strongnullstellensazt} and Remark~\ref{rem:nullst} applied to~\eqref{eq:assumption},  the system $F=0$ has no solutions in $K^{\mathbb{Z}}$.
Together with the existence of a partial solution of length $A + 1$, this  contradicts to Corollary~\ref{cor:620}.
 \end{proof}

\section*{Acknowledgments}
This work has been partially supported by NSF under grants (CCF-1564132, CCF-1563942, DMS-1760448, and DMS-1760413); PSC-CUNY under grant 60098-00 48;
and NSFC under grants (11688101, 11671014).
 
\bibliographystyle{abbrvnat}
\setlength{\bibsep}{6pt}
\bibliography{bibdata}
 
\end{document}